\documentclass{amsart}[12pt]
\usepackage[curve,matrix,arrow,frame,tips]{xy}
\usepackage{mathrsfs}
\usepackage{enumerate}
\usepackage{amssymb}
\usepackage{stmaryrd}
\usepackage[pdfborder={0 0 0}]{hyperref}

\DeclareMathOperator{\colim}{colim}
\DeclareMathOperator{\gl}{gl}
\DeclareMathOperator{\map}{map}
\DeclareMathOperator{\Id}{Id}
\DeclareMathOperator{\ev}{ev}
\DeclareMathOperator{\op}{op}
\DeclareMathOperator{\Rep}{Rep}

\newcommand{\mC}{{\mathbb C{}}}
\newcommand{\mN}{{\mathbb N}}
\newcommand{\mR}{{\mathbb R}}
\newcommand{\mZ}{{\mathbb Z}}
\newcommand{\Cc}{{\mathcal C}}

\newcommand{\Fc}{{\mathcal F}}
\newcommand{\Gc}{{\mathcal G}}

\newcommand{\Lc}{{\mathcal L}}
\newcommand{\Uc}{{\mathcal U}}
\newcommand{\Vc}{{\mathcal V}}
\newcommand{\Wc}{{\mathcal W}}
\newcommand{\Zc}{{\mathcal Z}}
\newcommand{\bL}{{\mathbf L}}
\newcommand{\bO}{{\mathbf O}}
\newcommand{\bT}{{\mathbf T}}
\newcommand{\spc}{spc}
\newcommand{\orbispc}{orbispc}
\newcommand{\iso}{\cong}
\newcommand{\tensor}{\otimes}
\newcommand{\xra}{\xrightarrow}
\newcommand{\xla}{\xleftarrow}
\newcommand{\GH}{{\mathcal{GH}}}
\newcommand{\td}[1]{\langle #1\rangle}
\newcommand{\gh}[1]{\llbracket #1\rrbracket}
\renewcommand{\to}{\longrightarrow}
\newcommand{\dslash}{\!\mathbin{/\mkern-11mu/}}

\numberwithin{equation}{section}
\newtheorem{theorem}[equation]{Theorem}
\newtheorem{prop}[equation]{Proposition}
 
\theoremstyle{definition}
\newtheorem{defn}[equation]{Definition}
\newtheorem{rk}[equation]{Remark}
\newtheorem{eg}[equation]{Example}
\newtheorem{construction}[equation]{Construction}

\begin{document}

\title[Orbispaces, orthogonal spaces, and the universal compact Lie group]
{Orbispaces, orthogonal spaces,\\ and the universal compact Lie group}

\date{\today; 2010 AMS Math.\ Subj.\ Class.: 55P91, 57S15}
\author{Stefan Schwede}
\address{Mathematisches Institut, Universit\"at Bonn, Germany}
\email{schwede@math.uni-bonn.de}

\begin{abstract}
This paper identifies the homotopy theories of topological stacks 
and orbispaces with unstable global homotopy theory.
At the same time, we provide a new perspective
by interpreting it as the homotopy theory of `spaces with an action
of the universal compact Lie group'.
The upshot is a novel way to construct and study 
genuine cohomology theories on stacks, orbifolds, and orbispaces,
defined from stable global homotopy types represented by orthogonal spectra.

The universal compact Lie group (which is neither compact nor a Lie group)
is a well-known object, namely the topological monoid $\mathcal L$
of linear isometric self-embeddings of $\mathbb R^\infty$.
The underlying space of $\mathcal L$ is contractible, and the homotopy theory
of $\mathcal L$-spaces with respect to underlying weak equivalences
is just another model for the homotopy theory of spaces.
However, the monoid $\mathcal L$ contains copies of all compact Lie
groups in a specific way, and we define global equivalences of $\mathcal L$-spaces
by testing on corresponding fixed points.
We establish a global model structure on the category of $\mathcal L$-spaces
and prove it to be Quillen equivalent
to the global model category of orthogonal spaces,
and to the category of orbispaces, i.e., presheaves of spaces
on the global orbit category.  
\end{abstract}

\maketitle

\section*{Introduction}

The purpose of this article is to identify the homotopy theory of topological stacks 
with unstable global homotopy theory.
At the same time, we provide a new perspective on this homotopy theory,
by interpreting it as the homotopy theory of `spaces with an action
of the universal compact Lie group'.
This link then provides a new way to construct and study 
genuine cohomology theories on stacks, orbifolds, and orbispaces.

Before describing the contents of this paper in more detail,
I expand on how global homotopy theory provides orbifold cohomology theories.
Stacks and orbifolds are concepts from algebraic geometry respectively
geometric topology that allow us to talk about objects
that locally look like the quotient of a smooth object by a group action,
in a way that remembers information about the isotropy groups of the action.
Such `stacky' objects can behave like smooth objects
even if the underlying spaces have singularities.
As for spaces, manifolds, and schemes, cohomology theories
are important invariants also for stacks and orbifolds,
and examples such as ordinary cohomology or $K$-theory lend themselves
to generalization. Special cases of orbifolds are `global quotients',
often denoted $M\dslash G$, for example for a smooth action of a compact
Lie group~$G$ on a smooth manifold~$M$. In such examples, the orbifold
cohomology of $M\dslash G$ is supposed to be the $G$-equivariant cohomology of~$M$.
This suggests a way to {\em define} orbifold cohomology theories by
means of equivariant stable homotopy theory, via suitable $G$-spectra.
However, since the group~$G$ is not intrinsic and can vary, one needs
equivariant cohomology theories for all groups~$G$, with some compatibility.
Global homotopy theory makes this idea precise.

As explained in~\cite{schwede-global}, the well-known 
category of orthogonal spectra models global stable
homotopy theory when studied with respect to {\em global equivalences},
the class of morphisms that induce isomorphisms of equivariant
stable homotopy groups for all compact Lie groups.
The localization is the {\em global stable homotopy category} $\GH$,
a compactly generated, tensor triangulated category that houses
all global stable homotopy types.
The present paper explains how to pass between stacks, orbispaces and
orthogonal spaces in a homotopically meaningful way; a consequence is
that every global stable homotopy type (i.e., every orthogonal spectrum)
gives rise to a cohomology theory on stacks and orbifolds.
Indeed, by taking unreduced suspension spectra,
every unstable global homotopy type is transferred into the triangulated
global stable homotopy category~$\GH$. In particular, taking morphisms
in~$\GH$ into an orthogonal spectrum~$E$ defines $\mZ$-graded 
$E$-cohomology groups. 
The counterpart of a global quotient orbifold~$M\dslash G$ in the global homotopy 
theory of orthogonal spaces is the semifree orthogonal space~$\bL_{G,V} M$
introduced in \cite[Con.\,1.1.22]{schwede-global}.
The morphisms $\gh{\Sigma^\infty_+ \bL_{G,V}M, E}^*$ in the global stable
homotopy category biject with the $G$-equivariant $E$-co\-homo\-logy groups of~$M$,
by an adjunction relating the global and~$G$-equivariant stable homotopy 
categories.
In other words, when evaluated on a global quotient~$M\dslash G$, 
our orbifold cohomology theory defined from a global stable homotopy type
precisely returns the $G$-equivariant cohomology of~$M$,
which is essentially the design criterion.
The cohomology theories defined in this way should be thought of as `genuine'
(as opposed to `naive').
Indeed, the global stable homotopy category forgets to
the $G$-equivariant stable homotopy category based on a complete $G$-universe;
the equivariant cohomology theories represented by such objects are
usually called {\em genuine} (as opposed to {\em naive}).
Genuine equivariant cohomology theories have more structure
than naive ones; this structure manifests itself in different forms, 
for example as transfer maps, stability under twisted suspension 
(i.e., smash product with linear representation spheres), 
an extension of the $\mZ$-graded cohomology groups to an $R O(G)$-graded theory,
and an equivariant refinement of additivity 
(the so called {\em Wirthm{\"u}ller isomorphism}).
Hence global stable homotopy types in the sense of \cite{schwede-global}
represent {\em genuine} (as opposed to `naive') orbifold cohomology theories.

\smallskip

Now we describe the contents of this paper in more detail.
There are different formal frameworks for 
stacks and orbifolds (algebro-geometric, smooth, topological), 
and these objects can be studied with respect to various notions of `equivalence'.
The approach that most easily feeds into our present context 
are the notions of {\em topological stacks}
respectively {\em orbi\-spaces} as developed by Gepner and Henriques
in \cite{gepner-henriques}. 
The present paper and the article \cite{koerschgen} by K{\"o}rschgen
together identify the Gepner-Henriques model 
with the global homotopy theory of orthogonal spaces
as established by the author in \cite[Ch.\,1]{schwede-global}.
The comparison proceeds through yet another model, the
global homotopy theory of `spaces with an action
of the universal compact Lie group'. 
This universal compact Lie group (which is neither compact nor a Lie group)
is a well-known object, namely 
the topological monoid $\Lc=\bL(\mR^\infty,\mR^\infty)$
of linear isometric embeddings of $\mR^\infty$ into itself.
To my knowledge, the use of the monoid~$\Lc$ in homotopy theory
started with Boardman and Vogt's paper~\cite{boardman-vogt-homotopy everything};
since then, $\Lc$-spaces have been extensively studied,
for example in~\cite{blumberg-cohen-schlichtkrull, EKMM, lind-diagram, may-quinn-ray}.
The underlying space of $\Lc$ is contractible, so the homotopy theory
of $\Lc$-spaces with respect to `underlying' weak equivalences
is just another model for the homotopy theory of spaces.
However, we shift the perspective on the homotopy theory that $\Lc$-spaces
represent, and use a notion of {\em global equivalences} 
of $\Lc$-spaces that is much finer than the notion of
underlying weak equivalence that has so far been studied. 

We will make the case that $\Lc$ has all the moral right to
be thought of as the universal compact Lie group.
Indeed, $\Lc$ contains a copy of every compact Lie group in a specific way: 
we may choose a continuous isometric linear $G$-action on $\mR^\infty$
that makes $\mR^\infty$ into a complete $G$-universe.
This action is a continuous injective homomorphism 
$\rho:G\to\Lc$, and we call the images $\rho(G)$ of such homomorphisms
{\em universal subgroups} of~$\Lc$, 
compare Definition~\ref{def-universal subgroup} below. 
In this way every compact Lie group determines a specific conjugacy class 
of subgroups of $\Lc$, abstractly isomorphic to $G$.
A morphism of $\Lc$-spaces is a {\em global equivalence}
if it induces weak homotopy equivalences on the fixed point
spaces for all universal subgroups of $\Lc$.
When viewed through the eyes of global equivalences,
one should think of an $\Lc$-space as a `global space' on which all compact Lie groups
act simultaneously and in a compatible way.

In Section~\ref{sec:L-spaces} we discuss the global homotopy theory
of $\Lc$-spaces, including many examples and the global model structure
(see Theorem \ref{thm:global L-spaces}).
We also discuss the operadic product of $\Lc$-spaces and show
that is fully homotopical for global equivalences
(see Theorem \ref{thm:box to times}) and satisfies the
pushout product property with respect to the global model structure
(see Proposition \ref{prop:global is monoidal}).

Section \ref{sec:L verses orbispaces} compares the global homotopy theory 
of $\Lc$-spaces to the homotopy theory of orbispaces. 
To this end we introduce the global orbit category~$\bO_{\gl}$
(see Definition \ref{def:O_gl}), the direct analog
for the universal compact Lie group
of the orbit category of a single compact Lie group:
the objects of $\bO_{\gl}$ are the universal subgroups of $\Lc$ and
the morphism spaces in~$\bO_{\gl}$ are defined by
\[ \bO_{\gl}(K,G)\ = \  \map^{\Lc}(\Lc/K,\Lc/G)  \ . \]
The main result is Theorem \ref{thm:L and orbispace}, describing a Quillen equivalence
between the category of $\Lc$-spaces under the global model structure and
the category of {\em orbispaces}, i.e., contravariant continuous functors
from the global orbit category to spaces.
This Quillen equivalence is an analog of Elmendorf's theorem~\cite{elmendorf-orbit} 
saying that taking fixed points with respect to all closed subgroups of $G$ 
is an equivalence from the homotopy theory of
$G$-spaces to functors on the orbit category.

Section \ref{sec:L vs orthogonal} compares 
the global homotopy theory of $\Lc$-spaces to the global homotopy theory
of orthogonal spaces as introduced in \cite[Ch.\,1]{schwede-global}.
The comparison is via an adjoint functor pair~$(Q\tensor_\bL-,\map^\Lc(Q,-))$ 
that was already used in a non-equivariant context (and with different terminology)
by Lind~\cite[Sec.\,8]{lind-diagram}.
We show in Theorem~\ref{thm:orthogonal vs L} 
that this functor pair is a Quillen equivalence with
respect to the two global model structures.

The following diagram of Quillen equivalences summarizes the results of this paper 
and the connection to the homotopy theory of orbispaces
in the sense of Gepner and Henriques \cite{gepner-henriques}:
\[\xymatrix@C=12mm{ 
 \text{orthogonal spaces}
  \quad \ar@<.4ex>[r]^-{Q\tensor_{\bL}-} 
 & 
\quad \Lc\bT\quad \ar@<-.4ex>[r]_-{\Phi} 
\ar@<.4ex>[l]^-{\map^\Lc(Q,-)} 
 & 
\ar@<-.4ex>[l]_-{\Lambda}
\quad   \text{orbispaces}  \ar@<.4ex>[r]\ar@<-.4ex>[r]
 \quad & 
 \quad \text{Orb-spaces} \ar@<.4ex>[l]\ar@<-.4ex>[l]
 }\]
On the left is the category of orthogonal spaces, i.e.,
continuous functors to spaces from the category 
of finite-dimensional inner product spaces and linear isometric embeddings,
equipped with the positive global model structure of \cite[Prop.\,1.2.23]{schwede-global}.
Next to it is the category of $\Lc$-spaces,
equipped with the global model structure of
Theorem \ref{thm:global L-spaces}.
Then comes the category of orbispaces in the sense of this paper,
i.e., contravariant continuous functors 
from the global orbit category $\bO_{\gl}$ to spaces;
orbispaces are equipped with a `pointwise' (or `projective') model structure.
Finally, on the right is the category of `Orb-spaces'
in the sense of Gepner and Henriques, i.e., contravariant continuous functors 
from the topological category Orb defined in \cite[Sec.\,4.1]{gepner-henriques},
for the family of compact Lie groups as allowed isotropy groups.
We establish the left and middle Quillen equivalence in 
Theorem \ref{thm:L and orbispace} respectively Theorem~\ref{thm:orthogonal vs L}. 
The right double arrows indicate a chain of two Quillen equivalences
established by K{\"o}rschgen in~\cite{koerschgen};
this chain arises from a zig-zag of Dwyer-Kan equivalences between 
our indexing category~$\bO_{\gl}$
and the category Orb used by Gepner and Henriques, see~\cite[1.1 Cor.]{koerschgen}.
In their paper \cite{gepner-henriques}, Gepner and Henriques furthermore
compare their homotopy theory of Orb-spaces 
to the homotopy theories of topological stacks and of topological groupoids.

{\bf Acknowledgments.} 
I would like to thank Andrew Blumberg, Benjamin B{\"o}hme, Mike Mandell, Peter May
and Alexander K{\"o}rschgen for various helpful discussions on the topics
of this paper.

\section{Global model structures for \texorpdfstring{$\Lc$}{L}-spaces}
\label{sec:L-spaces}

In this section we define global equivalences of $\Lc$-spaces
and establish the global model structure, see Theorem~\ref{thm:global L-spaces}.
In the following sections, we will show that this global model structure
is Quillen equivalent to the model category of orbispaces, compare
Theorem~\ref{thm:L and orbispace}, and to the category of orthogonal spaces, 
with respect to the positive global model structure,
compare Theorem~\ref{thm:orthogonal vs L}.

We denote by~$\Lc=\bL(\mR^\infty,\mR^\infty)$ the monoid, under composition,
of linear isometric self-embeddings of~$\mR^\infty$,
i.e., those~$\mR$-linear self-maps that preserve the scalar product.
The space $\Lc$ carries the compactly-generated function space topology.
The space $\Lc$ is contractible by
\cite[Sec.\ 1, Lemma]{boardman-vogt-homotopy everything} or Remark \ref{rk:L contractible},
and composition
\[ \circ \ : \ \Lc\times\Lc \ \to \ \Lc\]
is continuous, see for example Proposition \ref{prop:L acts continuously}.
Appendix \ref{app A} reviews the definition of the function space topology
in more detail, and collects
other point-set topological properties of spaces of linear isometric embeddings.

\begin{defn}
An {\em $\Lc$-space} is an $\Lc$-object in the category $\bT$ of
compactly generated spaces.
\end{defn}

More explicitly, an $\Lc$-space is a compactly generated space $X$
in the sense of \cite{mccord},
equipped with an action of the monoid~$\Lc$, and such that the action map
\[ \Lc\times X \ \to \ X \ , \quad (\varphi,x)\ \longmapsto \ \varphi\cdot x \]
is continuous for the Kelleyfied product topology on the source.
We write $\Lc\bT$ for the category of $\Lc$-spaces and $\Lc$-equivariant
continuous maps.

\begin{eg}
Examples of interesting $\Lc$-spaces are 
the global classifying spaces $\Lc/G$ for a compact Lie subgroup $G$ of
$\Lc$, see Example \ref{eg:L/G}. Many more examples arise from
orthogonal spaces via `evaluation at $\mR^\infty$' as explained in
Construction \ref{con:L acts in general}. The point of the Quillen equivalence
between $\Lc$-spaces and orthogonal spaces 
(Theorem \ref{thm:orthogonal vs L} in combination with Proposition \ref{prop:Q vs O})
is precisely that up to global equivalence, {\em all} $\Lc$-spaces
arise from orthogonal spaces by evaluation at $\mR^\infty$.

In particular, the orthogonal spaces  discussed in Chapters~1 and~2
of \cite{schwede-global} provide a host of examples of $\Lc$-spaces, and 
our Quillen equivalence translates all homotopical statements
about their global homotopy types into corresponding properties
of the associated $\Lc$-spaces. Explicit examples include
global projective spaces 
and Grassmannians \cite[Ex.\,1.1.28, 2.3.12--2.3.16]{schwede-global},
cofree global homotopy types \cite[Con.\,1.2.25]{schwede-global},
free ultra-commutative monoids \cite[Ex.\,2.1.5]{schwede-global},
global equivariant refinements $\mathbf{O}$, $\mathbf{S O}$, $\mathbf{U}$,
$\mathbf{S U}$, $\mathbf{S p}$, $\mathbf{Spin}$, $\mathbf{Spin}^c$
of the infinite families of classical Lie group \cite[Ex.\,2.3.6--2.3.10]{schwede-global},
unordered frames \cite[Ex.\,2.3.24]{schwede-global},
different global refinements $\mathbf{b O}$, $\mathbf{B O}$
and $\mathbf{B}^\circ\mathbf{O}$ of the classifying space of the infinite
orthogonal group \cite[Sec.\,2.4]{schwede-global},
global versions $\mathbf{BOP}$, $\mathbf{BUP}$ and $\mathbf{BSpP}$ of 
the infinite loop spaces of the real, complex and symplectic
equivariant $K$-theory spectra
\cite[Ex.\,2.4.1, 2.4.33]{schwede-global},
and the underlying `global infinite loop space' $\Omega^\bullet X$
of a stable global homotopy type $X$ \cite[Con.\,4.1.6]{schwede-global}.
\end{eg}

\begin{defn} 
Let $G$ be a compact Lie group. 
A {\em complete $G$-universe} is an orthogonal $G$-represen\-tation $\Uc$ 
of countably infinite dimension such that 
every finite-dimensional $G$-representation admits a $G$-equivariant linear isometric
embedding into $\Uc$.
\end{defn}

Proposition \ref{prop:compact subgroups}~(ii)
shows that an orthogonal representation of a compact Lie group $G$ on
a countably infinite dimensional inner product space
is necessarily the orthogonal direct sum of finite-dimensional $G$-invariant
subspaces. By further decomposing the summands into irreducible pieces,
the orthogonal $G$-represen\-tation can be written 
as the orthogonal direct sum of finite-dimensional irreducible orthogonal
$G$-representations. 
If the given representation is a complete $G$-universe, then
every irreducible $G$-representation must occur infinitely often.
So we conclude that a complete $G$-universe is 
equivariantly isometrically isomorphic to
\[\bigoplus_{\lambda\in\Lambda} \bigoplus_\mN\lambda \ ,\]
where $\Lambda=\{\lambda\}$ is a complete set of 
pairwise non-isomorphic irreducible $G$-representations.
The set $\Lambda$ is finite if $G$ is finite, and countably infinite if $G$ has
positive dimension.

\smallskip

Now we come to a key definition.

\begin{defn}\label{def-universal subgroup}
A {\em compact Lie subgroup} of the topological monoid $\Lc$ is a
subgroup that is compact in the subspace topology and admits the structure of a Lie group
(necessarily unique).
A compact Lie subgroup is a {\em universal subgroup}
if the tautological $G$-action makes~$\mR^\infty$ into a complete $G$-universe.
\end{defn}

Since $\Lc$ is a Hausdorff space, every compact subgroup is necessarily closed.
As we show in Proposition \ref{prop:characterize Lie},
a compact subgroup $G$ of $\Lc$ is a Lie group if and only if there
exists a finite-dimensional $G$-invariant subspace of $\mR^\infty$
on which $G$ acts faithfully. There are compact subgroups of $\Lc$
that are not Lie groups, compare Example \ref{eg:not Lie}.
When $G$ is a compact Lie subgroup of $\Lc$ we write $\mR^\infty_G$
for the tautological $G$-representation on~$\mR^\infty$,
which is automatically faithful.

The next proposition shows that conjugacy classes of 
universal subgroups of $\Lc$
biject with isomorphism classes of compact Lie groups.  

\begin{prop}\label{prop-uniqueness of universal subgroups}
Every compact Lie group is isomorphic to a universal subgroup of $\Lc$.
Every isomorphism between universal subgroups is given by conjugation
by an invertible linear isometry in~$\Lc$.
In particular, isomorphic universal subgroups are conjugate in~$\Lc$.
\end{prop}
\begin{proof}
  Given a compact Lie group $G$, we choose a continuous linear isometric action of~$G$ 
on $\mR^\infty$ that makes it a complete $G$-universe.
Such an action is adjoint to a continuous injective monoid homomorphism $\rho:G\to\Lc$.
Since~$G$ is compact and $\Lc$ is a Hausdorff space, the map~$\rho$
is a closed embedding, and the image $\rho(G)$ 
is a universal subgroup of $\Lc$, isomorphic to~$G$ via~$\rho$.

Now we let $\alpha:G\to \bar G$ be an isomorphism 
between two universal subgroups of~$\Lc$.
Then $\mR^\infty_G$ and $\alpha^*(\mR^\infty_{\bar G})$
are two complete $G$-universes, so there is a $G$-equivariant
linear isometric isomorphism $\varphi:\alpha^*(\mR^\infty_{\bar G})\to \mR^\infty_G$.
This~$\varphi$ is an invertible element of the monoid~$\Lc$ 
and the $G$-equivariance means that $\varphi\circ \alpha(g) = g\circ \varphi$
for all $g\in G$. Hence~$\alpha$ coincides with conjugation by $\varphi$.  
\end{proof}

The topological monoid $\Lc$ contains many other compact Lie subgroups
that are not universal subgroups: 
any continuous, faithful linear
isometric action of a compact Lie group $G$ on $\mR^\infty$
provides such a compact Lie subgroup. However, with respect to
this action, $\mR^\infty$ need not be a complete $G$-universe, because some irreducible
$G$-representations may occur only with finite multiplicity.

\begin{defn}
A morphism $f:X\to Y$ of $\Lc$-spaces is
a {\em global equivalence}
if for every universal subgroup~$G$ of~$\Lc$ the induced map
$f^G:X^G\to Y^G$ on $G$-fixed points is a weak homotopy equivalence.
\end{defn}

The class of global equivalences of $\Lc$-spaces 
is closed under various constructions; we collect some
of these properties in the next proposition.
We call a morphism $f:A\to B$ of $\Lc$-spaces an
{\em h-cofibration}
if it has the homotopy extension property, 
i.e., given a morphism of $\Lc$-spaces $\varphi:B\to X$ and a homotopy
of $\Lc$-equivariant maps $H:A\times [0,1]\to X$ such that $H(-,0)=\varphi f$,
there is a homotopy $\bar H:B\times [0,1]\to X$ such that
$\bar H\circ(f\times [0,1])=H$ and $\bar H(-,0)=\varphi$.
The class of h-cofibrations is closed under retracts, cobase change, coproducts 
and sequential compositions because it can be characterized by a right lifting property.

\begin{prop}\label{prop:global equiv basics} 
  \begin{enumerate}[\em (i)]
  \item 
    A coproduct of global equivalences is a global equivalence.
  \item 
    A product of global equivalences is a global equivalence. 
  \item Let $e_n:X_n\to X_{n+1}$ and~$f_n:Y_n\to Y_{n+1}$ be morphisms 
    of $\Lc$-spaces that are closed embeddings of underlying spaces, for $n\geq 0$. 
    Let~$\psi_n:X_n\to Y_n$ be global equivalences of $\Lc$-spaces
    that satisfy~$\psi_{n+1}\circ e_n=f_n\circ\psi_n$ for all $n\geq 0$.
    Then the induced morphism $\psi_\infty:X_\infty\to Y_\infty$ 
    between the colimits of the sequences is a global equivalence.
  \item Let $f_n:Y_n\to Y_{n+1}$ be a global equivalence 
    of $\Lc$-spaces
    that is a closed embedding of underlying spaces, for $n\geq 0$. 
    Then the canonical morphism 
    $f_\infty:Y_0\to Y_\infty$ to the colimit of the sequence $\{f_n\}_{n\geq 0}$
    is a global equivalence.
 \item  Let 
   \[ \xymatrix{
     C  \ar[d]_\gamma & A \ar[l]_-g \ar[d]^\alpha \ar[r]^-f & B \ar[d]^\beta\\
     \bar C  & \bar A \ar[l]^-{\bar g}\ar[r]_-{\bar f} & \bar B } \]
   be a commutative diagram of $\Lc$-spaces such that $g$ and~$\bar g$ are
   h-cofibrations.
   If the morphisms $\alpha,\beta$ and~$\gamma$ are global equivalences,
   then so is the induced morphism of pushouts
    \[ \gamma\cup \beta\ : \ C\cup_A B \ \to \ \bar C\cup_{\bar A} \bar B\ . \]
  \item  Let 
    \[ \xymatrix{ A \ar[r]^-f \ar[d]_g & B \ar[d]^h\\
      C \ar[r]_-k & D } \]
    be a pushout square of $\Lc$-spaces such that $f$ is a global equivalence.
    If in addition $f$ or $g$ is an h-cofibration, 
    then the morphism $k$ is a global equivalence.
 \end{enumerate}
\end{prop}
\begin{proof}
  Part~(i) holds because fixed points commute with disjoint unions
  and a disjoint union of weak equivalences is a weak equivalence.
  Part~(ii) holds because fixed points commute with products
  and a product of weak equivalences is a weak equivalence.
  
  Fixed points commute with sequential colimits along closed embeddings,
  compare \cite[Prop.\,B.1 (ii)]{schwede-global}.
  Moreover, a colimit of weak equivalences along two sequences of closed
  embeddings is another weak equivalence
  (compare \cite[Prop.\,A.17 (i)]{schwede-global}), so together this implies part~(iii). 
  Similarly, a sequential composite of weak equivalences that are simultaneously
  closed embeddings is another weak equivalence
  (compare \cite[Prop.\,A.17 (ii)]{schwede-global}), so together this implies part~(iv). 
  
  (v) Let $G$ be a universal subgroup of~$\Lc$. 
  Then the three vertical maps in the following commutative diagram 
  are weak equivalences:
  \[ \xymatrix@C=10mm{
    C^G  \ar[d]_{\gamma^G} &   A^G\ar[l]_-{g^G} \ar[d]^{\alpha^G} \ar[r]^-{f^G} & 
    B^G \ar[d]^{\beta^G}\\
    \bar C^G  &   \bar A^G \ar[l]^-{\bar g^G}\ar[r]_-{\bar f^G} &   \bar B^G } \]
  Since $g$ and $\bar g$ are h-cofibrations of $\Lc$-spaces, the
  maps $g^G$ and $\bar g^G$ are h-cofibrations of non-equivariant spaces.
  The induced map of the horizontal pushouts is thus a weak equivalence
  by the gluing lemma, 
  see for example~\cite[Appendix, Prop.\,4.8 (b)]{boardman-vogt-homotopy invariant}.
  Since $g^G$ and $\bar g^G$ are h-cofibrations, they are in particular 
  closed embeddings,
  compare \cite[Prop.\,8.2]{lewis-thesis} or \cite[Prop.\,A.31]{schwede-global}.
  So taking $G$-fixed points commutes with the horizontal pushout
  (compare \cite[Prop.\,B.1 (i)]{schwede-global}), and we conclude that
  also the map
  \[ (\gamma\cup \beta)^G\ : \ (C\cup_A B)^G \ \to \ (\bar C\cup_{\bar A} \bar B)^G \]
  is a weak equivalence. This proves part~(v).
  
  (vi) In the commutative diagram
  \[ \xymatrix{
    C  \ar@{=}[d] & A \ar[l]_-g \ar@{=}[d] \ar@{=}[r] & A \ar[d]^f\\
    C  & A \ar[l]^-{g}\ar[r]_-{f} & B } \]
  all vertical morphisms are global equivalences.
  If $g$ is an h-cofibration, we can apply part~(v) to this square
  to get the desired conclusion.
  If $f$ is an h-cofibration, we apply part~(v) after interchanging the roles
  of left and right horizontal morphisms.
\end{proof}

Several interesting morphisms of $\Lc$-spaces that come up in this paper 
satisfy the following stronger form of `global equivalence':

\begin{defn}
A morphism $f:X\to Y$ of $\Lc$-spaces is a {\em strong global equivalence}
if for every universal subgroup~$G$ of~$\Lc$, 
the underlying $G$-map of~$f$ is a $G$-equivariant homotopy equivalence.
\end{defn}

In other words, a morphism of $\Lc$-spaces $f:X\to Y$ is a strong global equivalence
if for every universal subgroup $G$ there is a $G$-equivariant
continuous map~$g:Y\to X$ such that $g\circ f:X\to X$
and $f\circ g:Y\to Y$ are $G$-equivariantly homotopic to the
respective identity maps. However, there is no compatibility 
requirement on the  homotopy inverses and the equivariant homotopies,
and they need not be equivariant for the full monoid $\Lc$.
Clearly, every strong global equivalence is in particular a global equivalence.

\begin{eg}[Induced $\Lc$-spaces and global classifying spaces]\label{eg:L/G}
We let $G$ be a compact Lie subgroup of $\Lc$ and $A$ a left $G$-space.
Then we can form the {\em induced $\Lc$-space}
\[ \Lc\times_G A \ = \ (\Lc\times A)/ (\varphi g,a)\sim (\varphi, g a)\ . \]
The functor $\Lc\times_G-$ is left adjoint to the restriction functor
from $\Lc$-spaces to $G$-spaces. 

The $G$-representation $\mR^\infty_G$ 
has a finite-dimensional faithful $G$-subrepresentation $V$,
by Proposition \ref{prop:characterize Lie}.
The following Proposition \ref{prop:L/G to L_G,V} shows in particular that 
for every universal subgroup $K$ of $\Lc$ the restriction map
\[  \Lc/G\ = \  \bL(\mR^\infty_G,\mR^\infty_K)/G \ \ \to \
\bL(V,\mR^\infty_K)/G \]
is a $K$-homotopy equivalence.
The $K$-space $\bL(V,\Uc_K)/G$ is a classifying space 
for principal $G$-bundles over $K$-spaces, 
see for example~\cite[Prop.\,1.1.30]{schwede-global}.
So under the tautological $K$-action, $\Lc/G$
is a classifying space for principal $G$-bundles over $K$-spaces.
We conclude that $\Lc/G$ is an incarnation, in the world of $\Lc$-spaces, 
of the global classifying space of the group $G$, or of the stack 
of principal $G$-bundles. 
In particular, the underlying non-equivariant space of $\Lc/G$ 
is a classifying space for $G$.
\end{eg}

\begin{prop}\label{prop:L/G to L_G,V} 
Let $G$ be a compact Lie subgroup of the monoid $\Lc$,
$V$ a faithful finite-dimensional 
$G$-subrepresen\-ta\-tion of $\mR^\infty_G$ and $A$ a $G$-space.
Then the restriction morphism
\[ \rho_V\times_G A \ : \ \Lc\times_G A \ = \ \bL(\mR_G^\infty,\mR^\infty)\times_G A \ \to \ 
\bL(V,\mR^\infty)\times_G A \]
is a strong global equivalence of $\Lc$-spaces.
\end{prop}
\begin{proof}
We let $K$ be a universal subgroup of the monoid $\Lc$.
Then $\rho_V: \bL(\mR_G^\infty,\mR_K^\infty)\to \bL(V,\mR^\infty_K)$ is a 
$(K\times G)$-homotopy equivalence by
Proposition~\ref{prop:EKG infinite}; so the map
\[ \rho_V \times_G A \ : \ \bL(\mR_G^\infty,\mR_K^\infty)\times_G A \ \to \ 
\bL(V,\mR^\infty_K)\times_G A \]
is a $K$-homotopy equivalence.
\end{proof}

Now we turn to the construction of model structures on the category of $\Lc$-spaces,
with the global model structure of Theorem \ref{thm:global L-spaces} 
as the ultimate aim.
The `classical' model structure on the category of all topological spaces
was established by Quillen in~\cite[II.3 Thm.\,1]{Q}.
We use the straightforward adaptation of this model structure to
the category $\bT$ of compactly generated spaces,
which is described for example in~\cite[Thm.\,2.4.25]{hovey-book}.
In this model structure, the weak equivalences 
are the weak homotopy equivalences, and fibrations
are the Serre fibrations. The cofibrations are the retracts of 
generalized CW-complexes, i.e., cell complexes in which cells
can be attached in any order and not necessarily to cells of lower dimensions.

Now we let $\Cc$ be a collection of closed subgroups of the 
linear isometries monoid $\Lc$.
We call a morphism $f:X\to Y$ of $\Lc$-spaces a {\em $\Cc$-equivalence}
(respectively {\em $\Cc$-fibration})
if the restriction $f^G:X^G\to Y^G$ to $G$-fixed points is a
weak equivalence (respectively Serre fibration)
of spaces for all subgroups $G$ that belong to the collection~$\Cc$.
A {\em $\Cc$-cofibration} is a morphism with the left lifting
property with respect to all morphisms that are simultaneously
$\Cc$-equivalences and $\Cc$-fibrations.
The resulting `$\Cc$-projective model structure'
would be well known if $\Lc$ were a topological group.
Despite the fact that $\Lc$ is only a topological monoid, 
the standard proof for topological groups
(see for example \cite[Prop.\,B.7]{schwede-global})
goes through essentially unchanged, and hence we omit it.

\begin{prop}\label{prop:projective model structure}
  Let $\Cc$ be a collection of closed subgroups of $\Lc$. 
  Then the $\Cc$-equivalences, $\Cc$-cofibrations and
  $\Cc$-fibrations form a model structure,
  the {\em $\Cc$-projective model structure}
  on the category of $\Lc$-spaces.
  This model structure is proper, cofibrantly generated and topological.
\end{prop}

For easier reference we recall the standard sets of generating cofibrations
and acyclic cofibrations of the $\Cc$-projective model structure.
We let $I_\Cc$ be the set of morphisms of $\Lc$-spaces
\begin{equation}\label{eq:define_I}
 \Lc/G\times \text{incl}\ : \ 
 \Lc/G\times \partial D^k\ \to \  \Lc/G\times D^k    
\end{equation}
for all $G$ in $\Cc$ all $k\geq 0$.
We let $J_\Cc$ denote the set of morphisms
\begin{equation}\label{eq:define_J}
 \Lc/G\times \text{incl}\ : \ 
 \Lc/G\times D^k\times \{0\}\ \to \ 
 \Lc/G\times D^k\times [0,1]  
\end{equation}
for all $G$ in $\Cc$ and all $k\geq 0$.
Then the right lifting property with respect to the set $I_\Cc$ 
(respectively $J_\Cc)$ is equivalent to being a $\Cc$-acyclic fibration
(respectively $\Cc$-fibration).

If we want a model structure on the category of $\Lc$-spaces with
the global equivalences as weak equivalences, then one possibility
is the {\em universal projective} model structure,
i.e., the $\Cc^{u}$-projective model structure of the previous Proposition 
\ref{prop:projective model structure},
for the collection $\Cc^{u}$ of universal
subgroups. We explain in Remark \ref{rk:global vs u-projective} 
below why this model structure is not the most convenient one for our purposes. 
We instead favor the following {\em global model structure},
which arises from the projective model structure
for the collection $\Cc^L$ of all compact Lie subgroups 
(as opposed to only universal subgroups)
by left Bousfield localization at the class of global equivalences.

\medskip

Let~$G$ and $\bar G$ be two compact Lie subgroups of $\Lc$
and  $\varphi\in\Lc$  a linear isometric embedding. 
We denote by
\[\text{stab}(\varphi) \ = \ \{ (g,\gamma)\in G\times\bar G \ : \ g\circ \varphi=\varphi\circ\gamma\}\]
the stabilizer group of $\varphi$ with respect to the action of
$G\times\bar G$ on $\Lc$ by post- respectively precomposition.
Since $\Lc$ is a Hausdorff space and composition of $\Lc$ is continuous,
the stabilizer $\text{stab}(\varphi)$ is a closed subgroup of $G\times\bar G$,
hence a compact Lie group in its own right.
The two projections from $\text{stab}(\varphi)$ to $G$ and $\bar G$
are continuous homomorphisms.

\begin{defn}
  Let~$G$ and $\bar G$ be two compact Lie subgroups of $\Lc$.
  A {\em correspondence} from $G$ to $\bar G$ is a linear isometric embedding 
  $\varphi\in\Lc$ such that the two projections 
  \[ G \ \xla{\quad}\ \text{stab}(\varphi) \ \xra{\quad} \ \bar G\]
  are isomorphisms.
\end{defn}

So a linear isometric embedding $\varphi\in\Lc$ 
is a correspondence from $G$ to $\bar G$ if and only if 
for every $g\in G$ there is a unique $\gamma\in\bar G$ such that
$g\circ \varphi=\varphi\circ\gamma$, and conversely
for every $\gamma\in \bar G$ there is a unique $g\in G$ such that
$g\circ \varphi=\varphi\circ\gamma$.
Since a linear isometric embedding is injective, the condition
$\varphi\circ\gamma = g\circ\varphi$ shows that $\gamma$
is determined by $\varphi$ and $g$. 

We write $\varphi:G\leadsto\bar G$ for a correspondence between compact Lie subgroups
of $\Lc$. Such a correspondence effectively embeds 
the tautological $\bar G$-representation
into the tautological $G$-representation. More precisely,
if $\alpha:G\to\bar G$ denotes the isomorphism provided by 
$\varphi$, i.e., $\varphi\circ\alpha(g)= g\circ\varphi$ 
for all $g\in G$, then $\varphi:\alpha^*(\mR^\infty_{\bar G})\to \mR^\infty_G$
is a $G$-equivariant linear isometric embedding.
If $\varphi$ happens to be bijective, then $\alpha(g)=\varphi^{-1}g\varphi$;
so informally speaking, one can think of a correspondence as an element
of $\Lc$ that `conjugates $G$ isomorphically onto $\bar G$'. The caveat is that
`conjugation by $\varphi$' does not have a literal meaning unless $\varphi$ is bijective.

Let $\varphi:G\leadsto\bar G$ be a correspondence between two compact Lie subgroups 
of $\Lc$. Then for every $g\in G$ there is a $\gamma\in\bar G$ such that
$g\circ\varphi = \varphi\circ \gamma$.
So for every $\Lc$-space $X$ and every $\bar G$-fixed element $x\in X^{\bar G}$,
we have
\[ g\cdot (\varphi\cdot x)\ = \ (g\circ\varphi)\cdot x
\ = \ (\varphi\circ \gamma)\cdot x\ = \ \varphi\cdot x\ .\]
Hence the continuous map $\varphi\cdot-: X\to X$ restricts to a map
\[ \varphi\cdot - \ : \ X^{\bar G}\ \to \ X^G  \]
from the $\bar G$-fixed points to the $G$-fixed points.
We denote by $\Cc^L$ the collection of compact Lie subgroups of $\Lc$.

\begin{defn}
  A morphism $f:X\to Y$ of $\Lc$-spaces is a {\em global fibration}
  if it is a $\Cc^L$-fibration
  and for every correspondence $\varphi:G\leadsto\bar G$ 
  between compact Lie subgroups of $\Lc$ the map
  \[ (f^{\bar G}, \varphi\cdot-)\ : \ X^{\bar G} \ \to \ 
  Y^{\bar G} \times_{Y^G }  X^G \]
  is a weak equivalence.
  An $\Lc$-space $X$ is {\em injective} if 
  for every correspondence $\varphi:G\leadsto\bar G$ 
  between compact Lie subgroups of $\Lc$ the map
  \[ \varphi\cdot - \ : \ X^{\bar G}\ \to \ X^G  \]
  is a weak equivalence.
\end{defn}

Equivalently, a morphism $f$ is a global fibration if and only if~$f$
is a $\Cc^L$-fibration and for every correspondence $\varphi:G\leadsto\bar G$ 
the square of fixed point spaces
\begin{equation}  \begin{aligned}
    \label{eq:fibration characterization}
    \xymatrix@C=15mm{ X^{\bar G} \ar[d]_{f^{\bar G}} \ar[r]^-{\varphi\cdot -} & 
      X^G \ar[d]^{f^G} \\
      Y^{\bar G} \ar[r]_-{\varphi\cdot -} & Y^G }
  \end{aligned}\end{equation}
is homotopy cartesian.  
Moreover, an $\Lc$-space is injective if and only if the unique morphism
to a terminal $\Lc$-space is a global fibration.

\begin{prop}\label{prop:level vs global fibration} 
  \begin{enumerate}[\em (i)]
  \item Every compact Lie subgroup of $\Lc$ admits a correspondence from a
    universal subgroup of $\Lc$.
  \item 
    Every global equivalence that is also a global fibration is a $\Cc^L$-equivalence.
  \item  
    Every global equivalence between injective $\Lc$-spaces is a $\Cc^L$-equivalence.
\end{enumerate}
\end{prop}
\begin{proof}
(i)
We let $\bar G$ be a compact Lie subgroup of $\Lc$. 
Proposition~\ref{prop-uniqueness of universal subgroups}
provides a universal subgroup $G$ of $\Lc$ 
and an isomorphism $\alpha:G\to \bar G$.
Since $\mR^\infty_G$ is a complete $G$-universe, 
there is a $G$-equivariant linear isometric embedding
$\varphi:\alpha^*(\mR^\infty_{\bar G})\to\mR^\infty_G$.
The fact that $\varphi$ is $G$-equivariant precisely means 
that $\varphi\circ\alpha(g) =  g\circ\varphi$ for all $g\in G$.
So $\varphi$ is a correspondence from $G$ to $\bar G$.

(ii) We let $f$ be a global equivalence of $\Lc$-spaces 
that is also a global fibration. We let $\bar G$ be any compact Lie subgroup
and $\varphi:G\leadsto\bar G$ a correspondence from a universal subgroup
as provided by part (i).
Then both vertical maps 
in the commutative square \eqref{eq:fibration characterization}
are Serre fibrations because $f$ is a $\Cc^L$-fibration.
The right vertical map is also a weak equivalence because $f$ is a global equivalence.
Since $f$ is a global fibration, the square is also homotopy cartesian,
so the left vertical map is a weak equivalence. Since $\bar G$ was any
compact Lie subgroup, this shows that $f$ is a $\Cc^L$-equivalence.

(iii) 
We let $f:X\to Y$ be a global equivalence between injective $\Lc$-spaces. 
We let $\bar G$ be any compact Lie subgroup
and $\varphi:G\leadsto\bar G$ a correspondence from a universal subgroup
as provided by part (i).
Then both horizontal maps 
in the commutative square \eqref{eq:fibration characterization}
are weak equivalences. 
The right vertical map is also a weak equivalence because $f$ is a global equivalence.
Hence the left vertical map is a weak equivalence. Since $\bar G$ was any
compact Lie subgroup, this shows that $f$ is a $\Cc^L$-equivalence.
\end{proof}

The sets $I=I_{\Cc^L}$ and $J=J_{\Cc^L}$ of generating cofibrations and acyclic cofibrations
for the $\Cc^L$-projective model structure were defined in \eqref{eq:define_I}
respectively  \eqref{eq:define_J}. We add another set of morphisms $K$ that detects
when the squares \eqref{eq:fibration characterization} are homotopy cartesian.
Given a correspondence $\varphi:G\leadsto\bar G$ 
between compact Lie subgroups of $\Lc$, the map
\[  \varphi_\sharp\ :\ \Lc/G \ \to \ 
 \Lc/\bar G  \ , \quad
\psi\cdot G \ \longmapsto \ (\psi \circ\varphi)\cdot \bar G \] 
is a morphism of $\Lc$-spaces that represents the natural transformation
\[ \varphi\cdot - \ : \ X^{\bar G}\ \to \ X^G  \ .\]

\begin{prop}
Let $\varphi:G\leadsto\bar G$ be a correspondence
between compact Lie subgroups of $\Lc$.
Then the morphism $\varphi_\sharp :\Lc/G\to \Lc/\bar G$
is a strong global equivalence of $\Lc$-spaces.
\end{prop}
\begin{proof}
Proposition \ref{prop:characterize Lie} provides a finite-dimensional faithful
$\bar G$-subrepresentation $V$ of~$\mR^\infty_{\bar G}$.
The space $\varphi(V)$ is then a finite-dimensional faithful
$G$-subrepresentation of $\mR^\infty_G$.
We obtain a commutative square of $\Lc$-spaces
\[ \xymatrix@C=20mm{ 
\bL(\mR^\infty_G,\mR^\infty)/G\ar[r]^-{\varphi_\sharp}\ar[d]_{\rho_{\varphi(V)}} &
\bL(\mR^\infty_{\bar G},\mR^\infty)/\bar G\ar[d]^{\rho_V} \\
\bL(\varphi(V),\mR^\infty)/G\ar[r]^\iso_-{\psi G\mapsto \psi(\varphi|_V)\bar G} &
\bL(V,\mR^\infty)/\bar G
 } \]
in which the vertical maps are restrictions.
Since $\varphi$ restricts to an isomorphism from $V$ to~$\varphi(V)$,
the lower horizontal map is an isomorphism.
Proposition \ref{prop:L/G to L_G,V} shows that the two vertical maps
are strong global equivalences of $\Lc$-spaces.
So the upper horizontal morphism is a strong global equivalence as well.
\end{proof}

We factor the global equivalence $\varphi_\sharp$ 
associated to a correspondence $\varphi:G\leadsto \bar G$
through the mapping cylinder as the composite
\[ \Lc/G\ \xra{\ c(\varphi)\ } \ Z(j) =
(\Lc/G\times[0,1])\cup_{\varphi_\sharp} \Lc/\bar G \ \xra{\ r(\varphi)\ } \  \Lc/\bar G \ ,\]
where $c(\varphi)$ is the `front' mapping  cylinder inclusion and~$r(\varphi)$ 
is the projection, which is a homotopy equivalence of~$\Lc$-spaces.
We then define $\Zc(\varphi)$
 as the set of all pushout product maps
\[  c(\varphi)\Box i_k \ : \  \Lc/G\times D^k \cup_{\Lc/G\times \partial D^k} 
Z(\varphi)\times \partial D^k\ \to \ Z(\varphi)\times D^k \]
for $k\geq 0$, where $i_k:\partial D^k\to D^k$ is the inclusion.
We then define 
\[K \ = \ \bigcup_{G,\bar G,\varphi} \ \Zc(\varphi) \ ,    \]
indexed by all triples consisting of two compact Lie subgroups
of $\Lc$ and a correspondence between them.
By \cite[Prop.\,1.2.16]{schwede-global}, the right lifting property 
with respect to the union $J\cup K$ then
characterizes the global fibrations, i.e., we have shown:

\begin{prop}\label{prop:global fibrations via RLP}
A morphism of $\Lc$-spaces is a global fibration if and only if it has
the right lifting property with respect to the set $J\cup K$.  
\end{prop}

Now we have all the ingredients to establish the global model structure
of $\Lc$-spaces. As the proof will show, the set $I=I_{\Cc^L}$
defined in \eqref{eq:define_I} is a set of generating cofibrations.
By Proposition \ref{prop:global fibrations via RLP},
the set $J\cup K=J_{\Cc^L}\cup K$ is a set of generating acyclic cofibrations.

\begin{theorem}[Global model structure for $\Lc$-spaces]\label{thm:global L-spaces} 
The $\Cc^L$-cofibrations, global fibrations and global equivalences form a 
cofibrantly generated proper topological model structure
on the category of $\Lc$-spaces, the {\em global model structure}.
The fibrant objects in the global model structure are the injective $\Lc$-spaces.
Every $\Cc^L$-cofibration is an h-cofibration of $\Lc$-spaces and
a closed embedding of underlying spaces.
\end{theorem}
\begin{proof}
We start with the last statement and let $f:A\to B$ be a $\Cc^L$-cofibration
of $\Lc$-spaces. 
For every $\Lc$-space $X$ the evaluation map
$\ev:X^{[0,1]}\to X$ sending a path $\omega$ to $\omega(0)$
is an acyclic fibration in the $\Cc^L$-projective model structure
of Proposition \ref{prop:projective model structure}.
Given a morphism $\varphi:B\to X$ and a homotopy
$H:A\times [0,1]\to X$ starting with $\varphi f$, we let $\hat H:A\to X^{[0,1]}$
be the adjoint and choose a lift in the commutative square:
\[ \xymatrix{ A\ar[d]_f\ar[r]^-{\hat H} & X^{[0,1]}\ar[d]^{\ev}\\
B\ar[r]_-\varphi\ar@{-->}[ur]^(.4)\lambda &X}  \]
The adjoint of the lift $\lambda$ is then the desired homotopy extending $\varphi$
and $H$. So the morphism $f$ is an h-cofibration.
Every h-cofibration of $\Lc$-spaces is in particular an h-cofibration
of underlying non-equivariant spaces, and hence a closed embedding
by \cite[Prop.\,8.2]{lewis-thesis} or \cite[Prop.\,A.31]{schwede-global}.

Now we turn to the model category axioms, where we use the numbering
as in~\cite[3.3]{dwyer-spalinski}. 
The category of $\Lc$-spaces is complete and cocomplete,
so axiom MC1 holds.
Global equivalences satisfy the 2-out-of-3 axiom MC2. 
Global equivalences, $\Cc^L$-cofibrations and global fibrations 
are closed under retracts, so axiom MC3 holds.

The $\Cc^L$-projective model structure 
of Proposition~\ref{prop:projective model structure}
shows that every morphism of $\Lc$-spaces
can be factored as $f\circ i$ for a $\Cc^L$-cofibration $i$
followed by a $\Cc^L$-equivalence $f$ that is also a $\Cc^L$-fibration.
For every correspondence $\varphi:G\leadsto\bar G$ 
between compact Lie subgroups, both vertical maps in the commutative square 
of fixed point spaces \eqref{eq:fibration characterization}
are then weak equivalences, so the square is homotopy cartesian.
The morphism $f$ is thus a global fibration and a global equivalence,
so this provides one of the factorizations as required by MC5.
For the other half of the factorization axiom MC5
we apply the small object argument
(see for example~\cite[7.12]{dwyer-spalinski} or~\cite[Thm.\,2.1.14]{hovey-book})
to the set $J\cup K$.
All morphisms in~$J$ are $\Cc^L$-cofibrations and $\Cc^L$-equivalences.
Since $\Lc/G$ is $\Cc^L$-cofibrant for every compact Lie subgroup $G$ of $\Lc$,
the morphisms in $K$ are also $\Cc^L$-cofibrations, 
and we argued above that they are global equivalences.
The small object argument provides a functorial factorization
of any given morphism of $\Lc$-spaces as a composite
\[ X \ \xra{\ i \ }\ W \ \xra{\ q \ }\ Y \]
where $i$ is a sequential composition of cobase changes of coproducts
of morphisms in~$J\cup K$, and~$q$ has the right lifting property with respect 
to $J\cup K$. 
Since all morphisms in $J\cup K$ are $\Cc^L$-cofibrations
and global equivalences, the morphism $i$ is a
$\Cc^L$-cofibration and a global equivalence,
using the various closure properties of the class of global equivalences
listed in Proposition~\ref{prop:global equiv basics}.
Moreover, $q$ is a global fibration by Proposition~\ref{prop:global fibrations via RLP}.

Now we show the lifting properties MC4. 
By Proposition~\ref{prop:level vs global fibration}~(ii)
a morphism that is both a global equivalence and a global fibration
is a $\Cc^L$-equivalence, and hence an acyclic fibration
in the $\Cc^L$-projective model structure. So every morphism that is
simultaneously a global equivalence and  a global fibration has the
right lifting property with respect to $\Cc^L$-cofibrations.
Now we let $j:A\to B$ be a $\Cc^L$-cofibration that is also a global equivalence and 
we show that it has the left lifting property with respect to all global fibrations.
We factor~$j=q\circ i$, via the small object argument for $J\cup K$,
where $i:A\to W$ is a $(J\cup K)$-cell complex and $q:W\to B$ a global fibration.
Then $q$ is a global equivalence since $j$ and~$i$ are,
and hence an acyclic fibration in the $\Cc^L$-projective model structure,
again by Proposition~\ref{prop:level vs global fibration}~(ii).
Since $j$ is a $\Cc^L$-cofibration, a lifting in
\[\xymatrix{
A \ar[r]^-i \ar[d]_j & W \ar[d]^q_(.6)\sim \\
B \ar@{=}[r] \ar@{..>}[ur] & B }\]
exists. Thus $j$ is a retract of the morphism~$i$ that has the left lifting property
with respect to global fibrations. But then $j$ itself has this lifting property.
This finishes the verification of the model category axioms.
Alongside we have also specified sets of generating cofibrations $I$
and generating acyclic cofibrations $J\cup K$.
Fixed points commute with sequential colimits along closed embeddings
(compare \cite[Prop.\,B.1 (ii)]{schwede-global})
and $\Cc^L$-cofibrations are closed embeddings.
So the sources and targets of all morphisms in $I$ and $J\cup K$ are small with
respect to sequential colimits of $\Cc^L$-cofibrations. So the global
model structure is cofibrantly generated.

Every $\Cc^L$-cofibration is in particular an h-cofibration of $\Lc$-spaces, 
and hence an h-cofibration of underlying $G$-spaces for every 
universal subgroup $G$.
So left properness follows from the gluing lemma for $G$-weak equivalences, compare
 \cite[Prop.\,B.6]{schwede-global}.
Since global equivalences are detected by fixed points with respect
to universal subgroups, 
right properness follows from right properness of the Quillen model structure
of spaces and the fact that fixed points preserve pullbacks.

The global model structure is topological by 
\cite[Prop.\,B.5]{schwede-global}, where we take $\mathcal G$
as the set of $\Lc$-spaces $\Lc/G$ for all compact Lie subgroups $G$ of $\Lc$,
and we take $\mathcal Z$ as the set of acyclic cofibrations
$c(\varphi):\Lc/G\to (\Lc/G\times[0,1])\cup_{\varphi_\sharp}\Lc/\bar G$
for $(G,\bar G,\varphi)$ as in the definition of the set~$K$.
\end{proof}

We end this section with a brief discussion on the interaction 
of the global model structure with the operadic product of $\Lc$-spaces.
We denote by
\[ \Lc(2)\ = \ \bL(\mR^\infty\oplus\mR^\infty,\,\mR^\infty) \]
the space of binary operations in the linear isometries operad.
It comes with a left action of~$\Lc$ and a right action of~$\Lc^2$ by
\[ \Lc\times \Lc(2) \times \Lc^2\ \to \ \Lc(2)\ , \quad
(f,\psi,(g,h)) \ \longmapsto \ f\circ\psi\circ(g\oplus h) \  . \]
Given two~$\Lc$-spaces~$X$ and~$Y$ we can coequalize 
the right~$\Lc^2$-action on~$\Lc(2)$
with the left~$\Lc^2$-action on the product~$X\times Y$ and form
\[  X\boxtimes_\Lc Y \ = \ \Lc(2) \times_{\Lc\times\Lc} (X\times Y)\ . \]
The left $\Lc$-action on~$\Lc(2)$ by postcomposition
descends to an~$\Lc$-action on this operadic product.
Some care has to be taken when analyzing this construction: because the monoid~$\Lc$
is not a group, it may be hard to figure out when
two elements of~$\Lc(2) \times X\times Y$ become equal in the coequalizer.
The operadic product~$\boxtimes_\Lc$ 
is coherently associative and commutative, but it does {\em not} have a unit object.
The monoids (respectively commutative monoids) with respect to~$\boxtimes_\Lc$
are essentially $A_\infty$-monoids (respectively $E_\infty$-monoids). 
We refer the reader to~\cite[Sec.\,4]{blumberg-cohen-schlichtkrull} for more details.

The next result shows that up to global equivalence
the operadic product of $\Lc$-spaces coincides with the categorical product.
Given two $\Lc$-spaces $X$ and $Y$, we define a natural~$\Lc$-linear map
\[ \rho_{X,Y}\ : \ X\boxtimes_\Lc Y \ \to\ X\times Y\text{\qquad by\qquad}
[\varphi;x,y]\ \longmapsto \ ( (\varphi i_1)\cdot x, (\varphi i_2)\cdot y)\ .\]
Here $i_1,i_2:\mR^\infty\to\mR^\infty\oplus\mR^\infty$ are the two direct summand embeddings.
The following theorem has a non-equivariant precursor:
Blumberg, Cohen and Schlichtkrull show 
in~\cite[Prop.\,4.23]{blumberg-cohen-schlichtkrull} 
that for certain $\Lc$-spaces 
(those that are cofibrant in the model structure of 
\cite[Thm.\,4.15]{blumberg-cohen-schlichtkrull}),
the morphism~$\rho_{X,Y}$ is a non-equivariant weak equivalence.
We show that a much stronger conclusion holds without any hypothesis on~$X$ and~$Y$.

\begin{theorem}\label{thm:box to times}
For all $\Lc$-spaces $X$ and $Y$, the morphism 
$\rho_{X,Y}:X\boxtimes_\Lc Y\to X\times Y$ is a strong global equivalence. 
In particular, the functor $X\boxtimes_\Lc -$ preserves global equivalences. 
\end{theorem}
\begin{proof}
We let~$G$ be a universal subgroup of~$\Lc$. We choose a $G$-equivariant
linear isometry
\[ \psi\ : \ \mR^\infty_G\oplus\mR^\infty_G \ \iso \ \mR^\infty_G \]
and define a continuous map
\[ \psi_* \ : \ X\times Y \ \to \ X\boxtimes_\Lc Y \text{\qquad by\qquad} 
\psi_*(x,y)\ = \ [\psi,x,y]\ .\]
The $G$-equivariance means explicitly that
$\psi(g\oplus g) =  g \psi$ for all~$g\in G$,
and so the map~$\psi_*$ is~$G$-equivariant (but {\em not} $\Lc$-linear).

The composite~$\rho_{X,Y}\circ\psi_*: X\times Y \to X\times Y$
is given by
\[ \rho_{X,Y}(\psi_*(x,y)) \ = \ ( (\psi i_1)\cdot x, (\psi i_2)\cdot y)\ .\]
Since $\mR^\infty_G$ is a complete $G$-universe, 
the space of $G$-equivariant linear isometric self-embeddings of~$\mR^\infty_G$
is contractible, see for example~\cite[II Lemma 1.5]{lms};
so there is a path of $G$-equivariant linear isometric self-embeddings
linking~$\psi i_1$ to the identity of~$\mR^\infty_G$.
Such a path induces a $G$-equivariant homotopy from the map 
$(\psi i_1)\cdot -:X\to X$ to the identity of~$X$;
similarly, $(\psi i_2)\cdot -$ is $G$-homotopic to the identity of~$Y$.
So altogether we conclude that $\rho_{X,Y}\circ \psi_*$
is $G$-homotopic to the identity.

To analyze the other composite we define a continuous map
\begin{align*}
 H \ : \  \bL(\mR^\infty\oplus\mR^\infty,\mR^\infty)\times [0,1] \ &\to \ 
\bL(\mR^\infty\oplus\mR^\infty,\mR^\infty\oplus\mR^\infty)\\
\text{by\qquad}
H(\varphi,t)(v,w)\ &= \ \left( \varphi(v, t w),\ \varphi(0, \sqrt{1-t^2}\cdot w) \right)\ .  
\end{align*}
Then
\[ H(\varphi,0)\ = \ (\varphi i_1)\oplus (\varphi i_2)
\text{\qquad and\qquad} 
H(\varphi,1) \ = \ i_1\varphi \ . \]
Moreover, for every~$t\in[0,1]$ the map $H(-,t)$ is equivariant for the
left~$\Lc$-action (with diagonal action on the target)
and for the right~$\Lc^2$-action.
So we can define a homotopy of $G$-equivariant maps
(which are {\em not} $\Lc$-linear)
\begin{align*}
 K \ : \ (X\boxtimes_\Lc Y)\times [0,1] \ &\to \ X\boxtimes_\Lc Y 
\text{\qquad by\qquad}
K([\varphi;x,y],t)\ = \ [ \psi H(\varphi,t);  x, y] \  .  
\end{align*}
Then 
\begin{align*}
 K([\varphi;x,y],0) \ = \ [ \psi H(\varphi,0); x, y] \
&= \ [ \psi ( (\varphi i_1) \oplus (\varphi i_2)); x, y ] \\
&= \ [\psi; (\varphi i_1) \cdot x, (\varphi i_2)\cdot  y] 
\ = \   \psi_*(\rho_{X,Y}[\varphi;x,y])
\end{align*}
and
\begin{align*}
 K([\varphi;x,y],1) \ &= \ [\psi H(\varphi,1); x, y ] \ 
= \ [  \psi i_1 \varphi; x, y ]  \  = \ (\psi i_1)\cdot  [ \varphi;  x, y ] \ .
\end{align*}
As in the first part of this proof,
$\psi i_1$ can be linked to the identity of~$\mR^\infty_G$
by a path of $G$-equivariant linear isometric self-embeddings,
and such a path induces another $G$-equivariant homotopy from the map 
$(\psi i_1)\cdot -:X\boxtimes_\Lc Y\to X\boxtimes_\Lc Y$ 
to the identity of~$X\boxtimes_\Lc Y$.
So altogether we have exhibited a $G$-homotopy
between~$\psi_*\circ\rho_{X,Y}$ and the identity.
Since the universal subgroup~$G$ was arbitrary,
this shows that~$\rho_{X,Y}$ is a strong global equivalence.
\end{proof}

The following `pushout product property' is the concise way to 
formulate the compatibility of the global model structure and operadic product.

\begin{prop}\label{prop:global is monoidal}
The global model structure on the category of $\Lc$-spaces 
satisfies the pushout product property with respect to the operadic box product:
for all $\Cc^L$-cofibrations of $\Lc$-spaces $i:A\to B$ and $j:X\to Y$,
the pushout product morphism
\[ i\Box j \ = \ (i\boxtimes_\Lc Y) \cup (B\boxtimes_\Lc j)\ : \ 
( A\boxtimes_\Lc Y) \cup_{A\boxtimes_\Lc X} (B\boxtimes_\Lc X) \ \to \ 
B\boxtimes_\Lc Y \]
is a $\Cc^L$-cofibration. If moreover $i$ or $j$ is a global equivalence,
then so is $i\Box j$.
\end{prop}
\begin{proof}
The key observation is the following. We let $G$ and~$K$ be compact Lie
groups and $\Vc$ respectively~$\Uc$ faithful orthogonal representations of~$G$
respectively~$K$ of countably infinite dimension.
Then the map
\begin{align*}
  \bL(\Vc,\mR^\infty) \boxtimes_\Lc \bL(\Uc,\mR^\infty) \ &\to \ 
\bL(\Vc\oplus\Uc,\mR^\infty) \ , \quad
[\varphi;\,\psi, \kappa ]\ \longmapsto \ 
\varphi\circ(\psi\oplus\kappa)
\end{align*}
is an isomorphism of $\Lc$-spaces by~\cite[I Lemma 5.4]{EKMM}
(sometimes referred to as `Hopkins' lemma').
The map is also $(G\times K)$-equivariant, and $\boxtimes_\Lc$
preserves colimits in both variables. So the map descends to
an isomorphism of $\Lc$-spaces
\begin{align*}
  \bL(\Vc,\mR^\infty)/G \boxtimes_\Lc \bL(\Uc,\mR^\infty)/K \ &\to \ 
\bL(\Vc\oplus\Uc,\mR^\infty)/(G\times K)\\
[\varphi;\,\psi G, \kappa K]\quad &\longmapsto \ 
(\varphi\circ(\psi\oplus\kappa))(G\times K)\ .
\end{align*}
On the other hand, $\Vc\oplus\Uc$ is a faithful orthogonal representation
for the group~$G\times K$.
We choose a linear isometry $\Vc\oplus\Uc\iso \mR^\infty$.
Conjugation with this isometry turns the action of $G\times K$
on  $\Vc\oplus\Uc$ into a continuous group monomorphism
$G\times K\to\Lc$; the image is thus a compact Lie subgroup $H\subset \Lc$
isomorphic to $G\times K$, and the operadic product
$ \Lc/G \boxtimes_\Lc \Lc/K$ is isomorphic to $\Lc/H$.

Since the operadic product preserves colimits in both variables,
it suffices to show the pushout product property for $\Cc^L$-cofibrations
in the generating set $I=I_{\Cc^L}$ for the global model structure,
compare \cite[Cor.\,4.2.5]{hovey-book}.
This set consists of the morphisms
\[ \Lc/G\times i_k \ : \ \Lc/G\times \partial D^k \ \to \ 
\Lc/G\times  D^k  \]
for all $k\geq 0$, where $G$ runs through all compact Lie subgroups of $\Lc$.
Since $i_k\Box i_m$ is isomorphic to $i_{k+m}$, the pushout product
\[ (\Lc/G\times i_k)\boxtimes_\Lc(\Lc/K\times i_m) \]
of two generating cofibrations is isomorphic
to $\Lc/H\times i_{k+m}$ for a compact Lie subgroup $H$,
and hence also a cofibration.

It remains to show that for every pair of $\Cc^L$-cofibrations $i:A\to B$ and $j:X\to Y$
such that $j$ is also a global equivalence, the pushout product morphism
is again a global equivalence.
The morphism $A\boxtimes_\Lc j :A\boxtimes_\Lc X \to A\boxtimes_\Lc Y$ 
is a global equivalence by Theorem~\ref{thm:box to times}.
Since $j$ is a $\Cc^L$-cofibration, it is also an h-cofibration
of $\Lc$-spaces (by Theorem~\ref{thm:global L-spaces}),
and hence  $A\boxtimes_\Lc j$ is again an h-cofibration. The cobase change 
\[ 
(B\boxtimes_\Lc X) \ \to \ ( A\boxtimes_\Lc Y) \cup_{A\boxtimes_\Lc X} (B\boxtimes_\Lc X) 
 \]
of $A\boxtimes_\Lc j$ is then a global equivalence by Proposition \ref{prop:global equiv basics}~(vi).  
The composite of this cobase change with the pushout product
morphism $i\Box j$ is the morphism $B\boxtimes_\Lc j$, and hence
a global equivalence by Theorem~\ref{thm:box to times}.
Hence $i\Box j$ is global equivalence.
\end{proof}

\begin{rk}[Global model structures for $\ast$-modules]
Since the unit transformation of the operadic product of $\Lc$-spaces
is not always an isomorphism, certain~$\Lc$-spaces are distinguished.
A {\em $\ast$-module} is an $\Lc$-space~$X$ for which the unit morphism 
\[  X\boxtimes_\Lc \ast \ \to\ X \ ,\qquad
[\varphi;x,\ast]\ \longmapsto \  (\varphi i)\cdot x \]
is an isomorphism, where $i:\mR^\infty\to\mR^\infty\oplus\mR^\infty$ 
is the embedding as the first direct summand.
The category of $\ast$-modules is particularly relevant
because on it, the one-point $\Lc$-space is a unit object for 
$\boxtimes_\Lc$ (by definition); so when restricted to 
the full subcategory of $\ast$-modules,  
the operadic product $\boxtimes_\Lc$ is symmetric monoidal.
One can show that for every orthogonal space~$A$,
the $\Lc$-space $A(\mR^\infty)$ 
(defined in Construction \ref{con:L acts in general} below) is a $\ast$-module, 
so these come in rich supply.
On the other hand, $\Lc$-spaces of the form~$\Lc/G$
for a compact Lie subgroup~$G$ of~$\Lc$ are {\em not}
$\ast$-modules.

The category of $\ast$-modules admits a (non-equivariant)
model structure with weak equivalences defined after forgetting
the $\Lc$-action, cf.~\cite[Thm.\,4.16]{blumberg-cohen-schlichtkrull}.
In~\cite{boehme}, B{\"o}hme constructs a monoidal model
structure on the category of $\ast$-modules that has the global equivalences
of ambient $\Lc$-spaces as its weak equivalences;
he also shows that with these global model structures, $\Lc$-spaces and $\ast$-modules
are Quillen equivalent, and that the global model structure on
$\ast$-modules lifts to associative monoids (with respect to $\boxtimes_\Lc$).
This effectively provides a global model structure on the category
of $A_\infty$-monoids, i.e., algebras over the linear isometries operad
(considered as a non-symmetric operad). 
It remains to be seen to what extent the global model structure
lifts to commutative monoids with respect to $\boxtimes_\Lc$ 
(i.e., to $E_\infty$-monoids).
\end{rk}

\section{\texorpdfstring{$\Lc$}{L}-spaces and orbispaces}
\label{sec:L verses orbispaces}

In this section we give rigorous meaning to the
slogan that global homotopy theory of $\Lc$-spaces
is the homotopy theory of `orbispaces with compact Lie group isotropy'.
To this end we formulate a version of Elmendorf's 
theorem~\cite{elmendorf-orbit} for the homotopy theory of $\Lc$-spaces,
saying that an $\Lc$-equivariant global homotopy type can be reassembled from
fixed point data.

\begin{defn}[Global orbit category]\label{def:O_gl}
The {\em global orbit category} $\bO_{\gl}$ is the topological category 
whose objects are all universal subgroups of the monoid $\Lc$, 
and the space of morphisms from~$K$ to $G$ is
\[ \bO_{\gl}(K,G)\ = \ \map^\Lc( \Lc/K,\Lc/G) \ , \]
the space of $\Lc$-equivariant maps from $\Lc/K$ to $\Lc/G$.
Composition in $\bO_{\gl}$ is composition of morphisms of $\Lc$-spaces.
\end{defn}

Since $\Lc/K$ represents the functor of taking $K$-fixed points,
the morphism space $\bO_{\gl}(K,G)$ is homeomorphic to
\[ (\Lc/G)^K \ = \ (\bL(\mR^\infty_G,\mR^\infty_K)/G)^K\ . \]

\begin{rk}
The global orbit category refines the category
$\Rep$ of compact Lie groups and conjugacy classes of continuous homomorphisms
in the following sense.
For all universal subgroups $G$ and $K$, the components
$\pi_0( \bO_{\gl}(K,G))$ biject functorially with $\Rep(K,G)$.
Indeed, by Proposition~\ref{prop:EKG infinite} 
the space $\bL(\mR^\infty_G,\mR^\infty_K)$ 
is $(K\times G)$-equivariantly homotopy equivalent to 
$\bL(V,\mR^\infty_K)$ for a finite-dimensional faithful $G$-representation $V$.
That latter space is a universal space for the family of graph subgroups,
compare \cite[Prop.\,1.1.26]{schwede-global}.
So the space $\bO_{\gl}(K,G)=(\Lc/G)^K$ is a 
disjoint union, indexed by conjugacy classes
of continuous group homomorphisms $\alpha:K\to G$,
of classifying spaces of the centralizer of the image of $\alpha$,
see for example
\cite[Prop.\,5]{lewis-may-mcclure-classifying} 
or~\cite[Prop.\,1.5.12 (i)]{schwede-global}.
In particular, the path component category $\pi_0(\bO_{\gl})$ of
the global orbit category is equivalent to the category $\Rep$
of compact Lie groups and conjugacy classes of continuous homomorphisms.
The preferred bijection
\[ \Rep(K,G)\ \to \ \pi_0( \bO_{\gl}(K,G))  \]
sends the conjugacy class of~$\alpha:K\to G$ to the $G$-orbit of any
$K$-equivariant linear isometric embedding of the $K$-universe $\alpha^*(\mR^\infty_G)$
into the complete $K$-universe~$\mR^\infty_K$.
\end{rk}

\begin{defn}
An {\em orbispace} is a continuous functor $Y:\bO_{\gl}^{\op}\to\bT$
from the opposite of the global orbit category to the category of spaces.
We denote the category of orbispaces and natural transformations by
$\orbispc$.
\end{defn}

For every small topological category $J$ with discrete object set
the category $J\bT$ of continuous functors from $J$ to spaces
has a `projective' model structure~\cite[Thm.\,5.4]{piacenza}
in which the weak equivalence and fibrations are those 
natural transformations that are weak equivalences respectively Serre fibrations
at every object. In the special case~$J=\bO_{\gl}^{\op}$, this provides
a projective (or objectwise) model structure on the category of orbispaces.

\begin{construction}
We introduce a {\em fixed point functor}
\[ \Phi \ : \  \Lc\bT \ \to \ \orbispc \]
from the category of $\Lc$-spaces to  the category of orbispaces 
that will turn out to be a right Quillen equivalence with respect to the
global model structure on the left hand side.
Given an $\Lc$-space $Y$ we define the value of the orbispace $\Phi(Y)$
at a universal subgroup $K$ as 
\[ \Phi(Y)(K)\ = \  \map^\Lc(\Lc/K,Y)\  ,\]
the space of $\Lc$-equivariant maps.
The $\bO_{\gl}$-functoriality is by composition of morphisms of $\Lc$-spaces.
In other words, $\Phi(Y)$ is the contravariant hom-functor represented by $Y$,
restricted to the global orbit category.
In particular,
\[ \Phi(\Lc/G)\ = \ \bO_{\gl}(-,G)\ , \]
i.e., the fixed points of the orbit $\Lc$-space $\Lc/G$ 
form the orbispace represented by~$G$.
Since $\Lc/K$ represents the functor of taking $K$-fixed points,
the map
\[ \Phi(Y)(K)\ \to \ Y^K \ , \quad f \ \longmapsto \ f(K) \]
is a homeomorphism.
\end{construction}

The fixed point functor~$\Phi$ has a left adjoint 
\[ \Lambda\ : \ \orbispc \ \to \ \Lc\bT \ ,\]
with value at an orbispace~$X$ given by the coend 
\[  \Lambda (X) \ = \ \int^{G\in \bO_{\gl}}\, \Lc/G \times X(G) \ , \]
i.e., a coequalizer, in the category of $\Lc$-spaces, of the two morphisms
\[  \xymatrix@C=10mm{  
\coprod_{K,G} \Lc/K \times \bO_{\gl}(K,G)\times X(G)
 \  \ar@<.4ex>[r] 
\ar@<-.4ex>[r] & \ \coprod_{G} \Lc/G\times X(G)  } \ .  \]
All we will need to know about the left adjoint is that for 
every universal subgroup~$G$ of~$\Lc$,
it takes  the representable orbispace $\bO_{\gl}(-,G)=\Phi(\Lc/G)$ to $\Lc/G$. 
Indeed, the counit $\epsilon_{\Lc/G}:\Lambda(\Phi(\Lc/G))\to \Lc/G$ induces a bijection
of morphism sets
\begin{align*}
  \Lc\bT(\Lambda(\Phi(\Lc/G)) , X)\ &\iso \  
  \orbispc(\Phi(\Lc/G), \Phi(X) )\\ &= \
  \orbispc(\bO_{\gl}(-,G), \Phi(X) )\ \iso \ \Phi(X)(G)\ = \    \Lc\bT(\Lc/G, X)\ .
\end{align*}
So the counit $\epsilon_{\Lc/G}:\Lambda(\Phi(\Lc/G))\to \Lc/G$ is an isomorphism of $\Lc$-spaces.

\begin{theorem} \label{thm:L and orbispace}
The adjoint functor pair
\[\xymatrix@C=12mm{ 
\Lambda \ : \ \orbispc  \quad \ar@<.4ex>[r] & 
\quad \Lc\bT\ : \ \Phi \ar@<.4ex>[l]  }\]  
is a Quillen equivalence between the category
of $\Lc$-spaces with the global model structure
and the category of orbispaces with the projective model structure.
Moreover, for every cofibrant orbispace $X$ the adjunction unit
$X\to\Phi(\Lambda X)$ is an isomorphism.
\end{theorem}
\begin{proof}
Every fibration in the global model structure of~$\Lc$-spaces
in particular restricts to a Serre fibration on fixed points
of every universal subgroup, so the right adjoint
sends fibrations in the global model structure of~$\Lc$-spaces
to fibrations in the projective model structure of orbispaces.
The right adjoint also takes global equivalences of~$\Lc$-spaces
to objectwise weak equivalences of orbispaces, by the very definition
of `global equivalences'. So~$\Phi$ is a right Quillen functor.

We now show that for every cofibrant orbispace $X$ the adjunction unit
$X\to\Phi(\Lambda X)$ is an isomorphism.
We let $\Gc$ denote the class of orbispaces for which the
adjunction unit is an isomorphism. We show the following property:
For every index set $I$, every $I$-indexed family $H_i$ 
of universal subgroups of~$\Lc$,
all numbers $n_i\geq 0$ and every pushout square of orbispaces 
\begin{equation}\begin{aligned}\label{eq:initial_pushout}
\xymatrix{ 
\coprod_{i\in I} \bO_{\gl}(-,H_i)\times \partial D^{n_i}\ar[r]\ar[d] &
\coprod_{i\in I} \bO_{\gl}(-,H_i)\times D^{n_i}  \ar[d] \\
X \ar[r] & Y }
 \end{aligned}\end{equation}
such that $X$ belongs to $\Gc$, the orbispace $Y$ also belongs to $\Gc$.

As a left adjoint, $\Lambda$ preserves pushout and coproducts.
For every space $A$ the functor $-\times A$ is a left adjoint, so it commutes
with colimits and coends. So $\Lambda$ also commutes with products with spaces. 
Thus $\Lambda$ takes the original square to a pushout square of $\Lc$-spaces:
\[  \xymatrix{ 
\coprod_{i\in I}  \Lc/H_i\times \partial D^{n_i}\ar[r]\ar[d] &
\coprod_{i\in I}  \Lc/H_i\times D^{n_i}  \ar[d] \\
\Lambda X \ar[r] & \Lambda Y }  \]
The upper horizontal morphism in this square is a closed embedding.
For every universal subgroup~$G$ of~$\Lc$ the $G$-fixed point functor
commutes with disjoint unions, products with spaces and
pushouts along closed embeddings. So the square
\[ \xymatrix{ 
\coprod_{i\in I} (\Lc/H_i)^G \times \partial D^{n_i}\ar[r]\ar[d] &
\coprod_{i\in I}  (\Lc/H_i)^G \times D^{n_i}  \ar[d] \\
(\Lambda X)^G \ar[r] & ( \Lambda Y)^G }\]
is a pushout in the category of compactly generated spaces.
Colimits and products of orbi\-spaces with spaces are formed objectwise, 
so letting $G$ run through all universal subgroups shows that the square
\[ \xymatrix{ 
\coprod_{i\in I} \Phi(\Lc/H_i) \times \partial D^{n_i}\ar[r]\ar[d] &
\coprod_{i\in I}  \Phi(\Lc/H_i) \times D^{n_i}  \ar[d] \\
\Phi(\Lambda X) \ar[r] &  \Phi(\Lambda Y) }\]
is a pushout in the category of orbispaces.
The adjunction units induce compatible maps from the
original pushout square~\eqref{eq:initial_pushout} to this last square.
Since $\Phi(\Lc/H_i)=\bO_{\gl}(-,H_i)$ and the unit $\eta_X:X\to \Phi(\Lambda X)$
is an isomorphism, the unit $\eta_Y:Y\to \Phi(\Lambda Y)$
is also an isomorphism.

The right adjoint~$\Phi$ preserves and detects all weak equivalences,
and for every cofibrant orbispace~$X$,
the adjunction unit $X\to\Phi(\Lambda X)$ is an isomorphism,
hence in particular a weak equivalence of orbispaces.
So the pair~$(\Lambda,\Phi)$ is a Quillen equivalence,
for example by~\cite[Cor.\,1.3.6]{hovey-book}. 
\end{proof}

\begin{rk}[Universal projective versus global model structure]\label{rk:global vs u-projective}
The Quillen equivalence of Theorem~\ref{thm:L and orbispace}  
can be factored as a composite of two composable Quillen equivalences as follows:
\[\xymatrix@C=10mm{ 
\orbispc  \quad \ar@<.4ex>[r]^-{\Lambda} & 
\quad (\Lc\bT)_{\text{u-proj}}\quad \ar@<.4ex>[l]^-{\Phi}  \ar@<.4ex>[r]^-{\Id} & 
\quad (\Lc\bT)_{\gl} \ar@<.4ex>[l]^-{\Id}  
}\]  
The middle model category is the category of $\Lc$-spaces, equipped with the
{\em universal projective} model structure, i.e.,
the projective model structure for the collection of universal subgroups. 
This intermediate model structure has the same weak equivalences
as the global model structure (namely the global equivalences),
but it has fewer cofibrations and more fibrations. 
The identity functor is thus a left Quillen equivalence from the
universal projective model structure to the global model structure
on the category of~$\Lc$-spaces.

Our reasons for emphasizing the global model structure 
(as opposed to the universal projective model structure)
for~$\Lc$-spaces are twofold. On the one hand, the global model
structure can be more easily compared to the global homotopy theory
of orthogonal spaces, as a certain adjoint functor 
already used by Lind in~\cite{lind-diagram} is a Quillen pair
for the global model structure on~$\Lc$-spaces
(but {\em not} for the universal projective model structure),
compare Theorem~\ref{thm:orthogonal vs L} below.
Another reason is that the global model structure of~$\Lc$-spaces
is monoidal with respect to the operadic $\boxtimes_\Lc$-product,
in the sense of Proposition~\ref{prop:global is monoidal} above.
The universal projective model structure, in contrast,
does {\em not} satisfy the pushout product property.
Indeed, for every universal subgroup~$G$ of~$\Lc$,
the $\Lc$-space $\Lc/G$ is cofibrant in the universal projective model structure.
If~$K$ is another universal subgroup of~$\Lc$,
then we showed in the proof of Proposition~\ref{prop:global is monoidal} 
that the operadic product
  \[ \Lc/G \boxtimes_\Lc \Lc/K \]
is isomorphic to $\Lc/H$ where $H$ is a compact Lie subgroup of~$\Lc$
isomorphic to $G\times K$. However, under the isomorphism
$H\iso G\times K$, the tautological action of~$H$ on~$\mR^\infty$
becomes the direct sum $\Uc_G\oplus\Uc_K$ 
of a complete $G$-universe and a complete $K$-universe.
While $\Uc_G\oplus\Uc_K$ is a universe for~$G\times K$, it is typically
{\em not} complete. So the operadic product $\Lc/G \boxtimes_\Lc \Lc/K$
is cofibrant in the global model structure, but typically {\em not}
in the universal projective model structure.
\end{rk}

\section{\texorpdfstring{$\Lc$}{L}-spaces and orthogonal spaces}
\label{sec:L vs orthogonal}

The aim of this section is to compare the global homotopy theory
of~$\Lc$-spaces to the global homotopy theory of orthogonal spaces
as developed by the author in~\cite{schwede-global}:
we will show in Theorem~\ref{thm:orthogonal vs L}
that the global model structure on the category of $\Lc$-spaces
is Quillen equivalent to 
the positive global model structure on the category of orthogonal spaces,
established in~\cite[Prop.\,1.2.23]{schwede-global}. 

We denote by $\bL$ the category with objects the finite-dimensional 
inner product spaces and morphisms the linear isometric embeddings.
If~$V$ and~$W$ are two finite-dimensional inner product spaces,
then the function space topology on $\bL(V,W)$ agrees with
the topology as the Stiefel manifold of $\dim(V)$-frames in~$W$,
by Proposition~\ref{prop:properties of L}~(i).
Moreover, composition of linear isometric embeddings is continuous, so
$\bL$ is then a topological category.
Every inner product space $V$ is isometrically isomorphic
to~$\mR^n$ with the standard scalar product, for $n$ the dimension of~$V$,
so the category~$\bL$ has a small skeleton.

\begin{defn} An {\em orthogonal space}
is a continuous functor $Y:\bL\to\bT$ to the category of spaces.
A morphism of orthogonal spaces is a natural transformation.
We denote by $\spc$ the category of orthogonal spaces.
\end{defn}

The category $\bL$ (or its extension that
also contains countably infinite dimensional inner product spaces)
is denoted $\mathscr I$ by Boardman and Vogt~\cite{boardman-vogt-homotopy everything},
and this notation is also used in~\cite{may-quinn-ray};
other sources~\cite{lind-diagram} use the symbol $\mathcal I$.
Accordingly, orthogonal spaces are sometimes referred to as $\mathscr I$-functors,
$\mathscr I$-spaces or $\mathcal I$-spaces.

\begin{construction}[Evaluation at $\mR^\infty$]\label{con:L acts in general}
We let~$Y$ be an orthogonal space. We extend the action maps
\begin{equation}  \label{eq:action_of_spc}
  \bL(V,W)\times Y(V)\ \to Y(W)  
\end{equation}
which are part of the structure of an orthogonal space to the 
situation where~$V$ and~$W$ are allowed to be of countably infinite dimension.
If~$\Wc$ is an inner product space of countably 
infinite dimension, then we let~$s(\Wc)$ denote 
the poset of finite-dimensional subspaces of~$\Wc$, ordered under inclusion.
We define
\[ Y(\Wc) \ = \ \colim_{W\in s(\Wc)}\, Y(W)\ ,\]
the colimit in the category~$\bT$ of compactly generated spaces.
If~$V$ is a finite-dimensional inner product space, 
we define the action map
\[ \bL(V,\Wc)\times Y(V) \ \to \ Y(\Wc)\]
from the action maps~\eqref{eq:action_of_spc}
of the functor~$Y$ by passing to colimits over the poset~$s(\Wc)$;
this is legitimate because $-\times Y(V)$ preserves colimits
and  $\bL(V,\Wc)$ is the colimit of the spaces $\bL(V,W)$
over the poset $s(\Wc)$,
by Proposition \ref{prop:properties of L}~(ii).

If~$\Vc$ is also of countably infinite dimension, then
$\bL(\Vc,\Wc)\times Y(\Vc)$
is the colimit of $\bL(\Vc,\Wc)\times Y(V)$ for~$V\in s(\Vc)$
because $\bL(\Vc,\Wc)\times -$ preserves colimits.
So the compatible maps
\[ \bL(\Vc,\Wc) \times Y(V) \ \xra{\rho^\Vc_V \times\Id} \ 
\bL(V,\Wc) \times Y(V) \ \xra{\text{act}} \ Y(\Wc)\]
assemble into a continuous action map.

Every orthogonal space~$Y$ gives rise to an~$\Lc$-space by evaluation at~$\mR^\infty$.
Indeed, for~$\Vc=\Wc=\mR^\infty$, the above construction
precisely says that the action maps make~$Y(\mR^\infty)$ into an~$\Lc$-space.  
If $V$ is a finite-dimensional inner product space, then
Proposition ~\ref{prop:properties of L}~(ii) shows that
$\bL(V,\mR^\infty)$ carries the weak topology with respect to
the closed subspaces $\bL(V,W)$ for $W\in s(\mR^\infty)$.
So 
\[ \bL(V,-)(\mR^\infty)\ = \ \bL(V,\mR^\infty) \ .\]
If $\Vc$ is an inner product space of countably infinite dimension,
then the space $\bL(\Vc,\mR^\infty)$
becomes an $\Lc$-space by postcomposition, but it  
does {\em not} arise from an orthogonal space by evaluation at~$\mR^\infty$. 
\end{construction}

We recall from~\cite[Def.\,1.1.2]{schwede-global}
the notion of {\em global equivalence} of orthogonal spaces.
We let~$G$ be a compact Lie group. 
For every orthogonal space~$Y$ and every 
finite-dimensional orthogonal $G$-representation $V$,
the value $Y(V)$ inherits a $G$-action from the
$G$-action on $V$ and the functoriality of $Y$. 
For a $G$-equivariant linear isometric embedding $\varphi:V\to W$
the induced map $Y(\varphi):Y(V)\to Y(W)$ is $G$-equivariant.

\begin{defn}
A morphism $f:X\to Y$
of orthogonal spaces is a {\em global equivalence}
if the following condition holds: for every compact Lie group $G$,
every $G$-representation~$V$, every $k\geq 0$ 
and all continuous maps $\alpha:\partial D^k\to X(V)^G$
and $\beta:D^k\to Y(V)^G$ such that $\beta|_{\partial D^k}=f(V)^G\circ\alpha$,
there is a $G$-representation~$W$, a $G$-equivariant linear
isometric embedding $\varphi:V\to W$ and a continuous map $\lambda:D^k\to X(W)^G$
such that $\lambda|_{\partial D^k}=X(\varphi)^G\circ \alpha$ and
such that $f(W)^G\circ \lambda$
is homotopic, relative to $\partial D^k$, to $Y(\varphi)^G\circ \beta$.
\end{defn}

In other words, for every commutative square on the left
\[ \xymatrix@C=13mm{
\partial D^k\ar[r]^-\alpha \ar[d]_{\text{incl}} & X(V)^G \ar[d]^{f(V)^G} &
\partial D^k\ar[r]^-\alpha \ar[d]_{\text{incl}} & X(V)^G \ar[r]^-{X(\varphi)^G} & 
X(W)^G \ar[d]^{f(W)^G} 
\\
D^k\ar[r]_-\beta & Y(V)^G & 
D^k\ar[r]_-\beta \ar@{-->}[urr]^-(.4)\lambda & Y(V)^G \ar[r]_-{Y(\varphi)^G} & Y(W)^G  }\]
there exists the lift $\lambda$ on the right hand side that makes the upper left
triangle commute on the nose, and the lower right triangle up to homotopy
relative to~$\partial D^k$.

An orthogonal space is {\em closed} if all its structure maps are closed
embeddings. The following is a reformulation of \cite[Prop.\,1.1.17]{schwede-global}.

\begin{prop}
 A morphism $f:X\to Y$ between closed orthogonal spaces is a global equivalence 
if and only if 
$f(\mR^\infty):X(\mR^\infty)\to Y(\mR^\infty)$ is a global equivalence of $\Lc$-spaces.
\end{prop}

Proposition~1.2.23 of~\cite{schwede-global}
establishes the positive global model structure on the category of orthogonal spaces
in which the weak equivalences are the global equivalences.
A morphism $f$ is a `positive global fibration'
(i.e., a fibration in the positive global model structure) 
if and only if for every compact Lie group $G$,
every faithful $G$-representation $V$ with $V\ne 0$ and 
every equivariant linear isometric embedding~$\varphi:V\to W$ 
the map~$f(V)^G:X(V)^G\to Y(V)^G$ is a Serre fibration and
the square of $G$-fixed point spaces
\[ 
        \xymatrix@C=18mm{ X(V)^G \ar[d]_{f(V)^G} \ar[r]^-{X(\varphi)^G} & 
           X(W)^G \ar[d]^{f(W)^G} \\
          Y(V)^G \ar[r]_-{Y(\varphi)^G} & Y(W)^G } \]
    is homotopy cartesian.  
This positive global model structure is proper, topological, compactly generated
and monoidal with respect to the convolution box product
of orthogonal spaces.

To compare the global model structures of $\Lc$-spaces and orthogonal spaces
we use the adjoint functor pair 
\[\xymatrix@C=10mm{ Q\tensor_{\bL}- \ : \ \spc \quad \ar@<.4ex>[r] & 
\quad \Lc\bT \ : \ \map^\Lc(Q,-) \ar@<.4ex>[l] }\]
introduced by Lind in~\cite[Sec.\,8]{lind-diagram};
Lind denotes the functor $Q\tensor_{\bL}-$ by $\mathbb Q$.
The adjoint pair arises from a continuous functor
\[ Q\ : \ \bL^{\op} \ \to \ \Lc\bT \ , \quad V \ \longmapsto \ 
\bL(V\tensor\mR^\infty,\mR^\infty)\ .\]
Here $\Lc$ acts on~$\bL(V\tensor\mR^\infty,\mR^\infty)$
by postcomposition. A linear isometric embedding~$\varphi:V\to W$
induces the morphism of $\Lc$-spaces
\[ Q(\varphi)\ = \ \bL(\varphi\tensor\mR^\infty,\mR^\infty) \ : \ 
\bL(W\tensor\mR^\infty,\mR^\infty)\ \to \ 
\bL(V\tensor\mR^\infty,\mR^\infty)\ ,\quad \psi \ \longmapsto \ 
\psi\circ(\varphi\tensor\mR^\infty)\ .\]
Since orthogonal spaces are defined as the continuous functors from~$\bL$,
and since the category of $\Lc$-spaces is tensored and cotensored over spaces,
any continuous functor from $\bL^{\op}$ induces an adjoint functor pair 
by an enriched end-coend construction.
Indeed, the value of the left adjoint on an orthogonal space~$Y$ is given by
\[ Q\tensor_{\bL} Y \ = \ \int^{V\in\bL} Y(V)\times \bL(V\tensor \mR^\infty,\mR^\infty) \ ,\]
the enriched coend of the continuous functor
\[ \bL\times \bL^{\op} \ \to \ \Lc\bT\ , \quad (V,W)\ \longmapsto \ 
Y(V)\times \bL(W\tensor \mR^\infty,\mR^\infty)
 \ .\]
The functor~$Q\tensor_{\bL}-$ has a right adjoint $\map^\Lc(Q,-)$
whose value at an $\Lc$-space~$Z$ is given by
\[ \map^\Lc(Q,Z)(V)\ = \ \map^\Lc(Q(V),Z)\ , \]
the mapping space of $\Lc$-equivariant maps from~$Q(V)$ to~$Z$.
The covariant functoriality in $V$ comes from the contravariant functoriality of~$Q$.
The coend of a contravariant functor with a representable covariant functor
returns the value at the representing object, i.e.,
\[  Q\tensor_{\bL} \bL(V,-)\ = \ Q(V) \ = \ \bL(V\tensor\mR^\infty,\mR^\infty)\ . \]
So the value on the semifree orthogonal space~$\bL_{G,V}=\bL(V,-)/G$ generated by
a $G$-represen\-ta\-tion $V$ comes out as
\[ Q\tensor_{\bL} \bL_{G,V}\ \iso \ ( Q\tensor_{\bL} \bL(V,-) ) / G\ \iso \ Q(V)/G \ = \ 
\bL(V\tensor\mR^\infty,\mR^\infty) / G \ .   \]
By \cite[Lemma~8.3]{lind-diagram}, 
the functor~$Q\tensor_\bL -$ from orthogonal spaces to $\Lc$-spaces is 
strong symmetric monoidal for the box product of orthogonal spaces
(compare \cite[Sec.\,1.3]{schwede-global}) 
and the operadic product $\boxtimes_\Lc$ of $\Lc$-spaces.

\begin{prop}\label{prop:map(Q,Y) and fix}
Let $Y$ be an injective $\Lc$-space, $G$ a compact Lie subgroup of $\Lc$
and $V$ a non-zero finite-dimensional faithful $G$-representation.
Then for every $G$-equivariant linear isometric embedding
$\kappa:V\tensor\mR^\infty\to\mR^\infty_G$ the $G$-equivariant evaluation map
\[ \map^\Lc(Q,Y)(V) \ = \ 
\map^\Lc(\bL(V\tensor\mR^\infty,\mR^\infty),Y) \ \to \ Y \ , \quad
h\mapsto h(\kappa) \]
is a $G$-weak equivalence.
\end{prop}
\begin{proof}
We show that the evaluation map restricts to a weak equivalence
on $G$-fixed points.  
Since the hypotheses are stable under passage to closed subgroups of $G$,
applying this to closed subgroups provides the claim.

We choose a linear isometry $\psi:V\tensor\mR^\infty\iso \mR^\infty$
and transport the $G$-action on  $V\tensor\mR^\infty$ to
a $G$-action on $\mR^\infty$ by conjugating with $\psi$.
Since $G$ is compact and $V$ is faithful, the conjugation homomorphism
\[ c_\psi\ : \  G \ \to \ \Lc \]
is then a closed embedding, and hence an isomorphism onto its image
$\bar G=c_\psi(G)$.
In particular,  $\bar G$ is a compact Lie subgroup of~$\Lc$.
Moreover,
\[ (\kappa\psi^{-1})\circ c_\psi(g)\ = \ 
(\kappa\psi^{-1})\circ (\psi g\psi^{-1}) \ = \ 
\kappa g\psi^{-1} \ = \ g (\kappa \psi^{-1})  \]
for all $g\in G$, so $\kappa\psi^{-1}:G\leadsto\bar G$ is
a correspondence.
Evaluation at $\kappa$ factors as the composite
\[  
 \left(\map^\Lc(Q,Y)(V)\right)^G \ \xra[\iso]{h\mapsto h(\psi)} \
Y^{\bar G} \ \xra{(\kappa\psi^{-1})\cdot -} \ Y^G  \ ;\]
the first map is a homeomorphism,
and the second map is a weak equivalence because $Y$ is injective.
So evaluation at $\kappa$ is a weak equivalence.
\end{proof}

The functor $Q\tensor_\bL -$ is closely related to evaluation
of an orthogonal space at $\mR^\infty$.
Indeed, a choice of unit vector $u\in \mR^\infty$  gives rise to a 
natural linear isometric embedding~$-\tensor u:V\to V\tensor\mR^\infty$.
As $V$ ranges over all finite-dimensional inner product spaces, the composite maps
\[  Y(V) \times \bL(V\tensor \mR^\infty,\mR^\infty)\ \xra{Y(V)\times\bL(-\tensor u,\mR^\infty)}\ Y(V) \times \bL(V,\mR^\infty)\ \xra{\text{act}}\ Y(\mR^\infty) \]
are compatible with the coend relations, and assembly into a natural map
\[ \xi_Y \ : \  Q\tensor_\bL Y\ \to \ Y(\mR^\infty)\ . \]
An orthogonal space $Y$ is {\em flat} if a specific latching map $\nu_m:L_m Y\to Y(\mR^m)$
is an $O(m)$-cofibration for every $m\geq 0$, compare \cite[Def.\,1.2.9]{schwede-global};
flat orthogonal spaces are in particular closed by \cite[Prop.\,1.2.11]{schwede-global},
and they are the cofibrant objects in the global model structure
of  \cite[Thm.\,1.2.21]{schwede-global}.
Lind shows in~\cite[Lemma 9.7]{lind-diagram} that for every flat orthogonal
space (i.e., cofibrant~$\mathcal I$-space in his terminology)
the map~$\xi_Y:Q\tensor_\bL Y\to Y(\mR^\infty)$ is a weak equivalence.
We generalize this as follows:

\begin{prop}\label{prop:Q vs O}
For every flat orthogonal space $Y$
the map $\xi_Y: Q\tensor_\bL Y\to Y(\mR^\infty)$ is a global equivalence of $\Lc$-spaces.
\end{prop}
\begin{proof}
We start with the special case where $Y$ is one
of the generating objects for the flat cofibrations. In other words, we
let $Y=\bL_{G,V}=\bL(V,-)/G$ be the semifree orthogonal space for a compact
Lie group $G$ and a faithful $G$-representation~$V$.
In this case, the $\Lc$-space $Q\tensor_\bL \bL_{G,V}$ is isomorphic to
$\bL(V\tensor\mR^\infty,\mR^\infty)/G$ and 
$\bL_{G,V}(\mR^\infty)$ is isomorphic to
$\bL(V,\mR^\infty)/G$. Under these identifications the morphism $\xi_{\bL_{G,V}}$ becomes
the restriction morphism
\[ \bL(-\tensor u,\mR^\infty)/G \ : \ 
\bL(V\tensor\mR^\infty,\mR^\infty)/G\ \to \ \bL(V,\mR^\infty)/G  \]
induced by the $G$-equivariant linear isometric embedding
$-\tensor u:V\to V\tensor\mR^\infty$.
If $V=0$, then this morphism is an isomorphism of $\Lc$-spaces.
If $V\ne 0$, then for every universal subgroup $K$ of $\Lc$, the restriction map
\[ \bL(-\tensor u,\mR^\infty_K) \ : \ 
\bL(V\tensor\mR^\infty,\mR^\infty_K)\ \to \ \bL(V,\mR^\infty_K)  \]
is a $(K\times G)$-homotopy equivalence by Proposition \ref{prop:EKG infinite}.
So the map descends to a $K$-homotopy equivalence on $G$-orbit spaces,
and $\bL(-\tensor u,\mR^\infty)/G$ is a global equivalence of $\Lc$-spaces.
This proves the special case $Y=\bL_{G,V}$.

Now we suppose that the orthogonal space $Y$ arises as the colimit of a sequence
\begin{equation}\label{eq:spc sequence}
 \emptyset = Y_0 \ \to \ Y_1 \ \to \dots \ \to \ Y_n \ \to \ \dots 
\end{equation}
in which each~$Y_n$ is obtained from~$Y_{n-1}$ as a pushout of orthogonal spaces
\[\xymatrix{
 \coprod_I \bL_{G_i,V_i}\times \partial D^{k_i} \ar[r]^-{\text{incl}}\ar[d] &
 \coprod_I \bL_{G_i,V_i}\times  D^{k_i} \ar[d] \\
Y_{n-1}\ar[r] & Y_n}\]
for some indexing set~$I$, compact Lie groups $G_i$, 
faithful $G_i$-representations $V_i$, and numbers $k_i\geq 0$.
We show by induction that the morphisms $\xi_{Y_n}: Q\tensor_\bL Y_n\to Y_n(\mR^\infty)$ 
are global equivalences. The induction starts with $n=0$, where there is nothing to show.

For the inductive step we assume that the morphism $\xi_{Y_{n-1}}$ is a global equivalence,
and we show that then $\xi_{Y_n}$ is a global equivalence as well.
The functors $Q\tensor_\bL-$ and evaluation at $\mR^\infty$ preserve
colimits. The morphism $\xi_{Y_n}$ is thus induced by the commutative diagram
\[  \xymatrix{
Q\tensor_\bL Y_{n-1} \ar[d]_{\xi_{Y_{n-1}}}& 
\coprod_I Q\tensor_\bL \bL_{G_i,V_i}\times \partial D^{k_i} 
\ar[r]^-{\text{incl}}\ar[l]\ar[d]^{\coprod \xi_{\bL_{G_i,V_i}}\times \partial D^{k_i}}
 &
 \coprod_I Q\tensor_\bL \bL_{G_i,V_i}\times  D^{k_i} \ar[d]^{\coprod \xi_{\bL_{G_i,V_i}}\times D^{k_i}} \\
Y_{n-1}(\mR^\infty) & \coprod_I \bL_{G_i,V_i}(\mR^\infty)\times \partial D^{k_i} \ar[r]_-{\text{incl}}\ar[l] &
 \coprod_I \bL_{G_i,V_i}(\mR^\infty)\times  D^{k_i} }\]
by passage to pushouts in the horizontal direction.
The left vertical map is a global equivalence by hypothesis.
The middle and right vertical maps are global equivalences by the
special case above and parts (i) and (ii) of Proposition \ref{prop:global equiv basics}.
The inclusion of $ \coprod_I \bL_{G_i,V_i}\times \partial D^{k_i}$ into 
$ \coprod_I \bL_{G_i,V_i}\times  D^{k_i}$ is an h-cofibration of orthogonal spaces.
Since both $Q\tensor_\bL-$ and evaluation at $\mR^\infty$ preserve h-cofibrations,
the two right horizontal morphisms in the diagram are h-cofibrations
of $\Lc$-spaces.
So the induced map on pushouts $\xi_{Y_n}$ is a global equivalence of $\Lc$-spaces
by Proposition \ref{prop:global equiv basics}~(v).

Since $Y$ is a colimit of the sequence~\eqref{eq:spc sequence},
$Q\tensor_\bL Y$ is a colimit of the sequence~$Q\tensor_\bL Y_n$,
and $Y(\mR^\infty)$ is a colimit of the sequence~$Y_n(\mR^\infty)$.
Moreover, since each morphism $Y_{n-1}\to Y_n$ is an h-cofibration
of orthogonal spaces, the morphisms
 $Q\tensor_\bL Y_{n-1}\to Q\tensor_\bL Y_n$
and $Y_{n-1}(\mR^\infty)\to Y_n(\mR^\infty)$ are h-cofibrations of $\Lc$-spaces,
and hence closed embedding.
Since global equivalences are homotopical for 
sequential colimits along closed embeddings 
(by Proposition  \ref{prop:global equiv basics}~(iii)),
we conclude that the morphism $\xi_Y$ is a global equivalence.

Every flat orthogonal space is a retract of a flat orthogonal space
of the form considered in the previous paragraph.
Since global equivalences are closed under retracts, 
the morphism $\xi_Y$ is a global equivalence for every flat orthogonal space.
\end{proof}

\begin{theorem}\label{thm:orthogonal vs L}
The adjoint functor pair
\[\xymatrix@C=10mm{
Q\tensor_\bL-\ : \  \spc \quad \ar@<.4ex>[r] & 
\quad \Lc\bT \ : \ \map^\Lc(Q,-)\ar@<.4ex>[l]  }\]
is a Quillen equivalence with respect to the positive global model structure
on orthogonal spaces and the  global model structure on $\Lc$-spaces.  
\end{theorem}
\begin{proof}
This theorem is a global sharpening of Lind's non-equivariant
Quillen equivalence \cite[Thm.\ 9.9]{lind-diagram}, 
and some of our arguments are `globalizations' of Lind's.
We let $G$ be a compact Lie group and $V$ a non-trivial faithful $G$-representation.
Then $V\tensor\mR^\infty$ is a faithful $G$-representation of countably infinite
dimension.
We choose a linear isometry $\psi:V\tensor\mR^\infty\iso \mR^\infty$
and transport the $G$-action on  $V\tensor\mR^\infty$ to
a $G$-action on $\mR^\infty$ by conjugating with $\psi$.
Since $G$ is compact, the conjugation homomorphism
\[ c_\psi\ : \  G \ \to \ \Lc \]
is then a closed embedding, and hence an isomorphism onto its image
$\bar G=c_\psi(G)$, which is then a compact Lie subgroup of~$\Lc$.
Moreover, $\psi:V\tensor\mR^\infty\iso c_\psi^*(\mR^\infty_{\bar G})$
is $G$-equivariant, by construction.
Evaluation at $\psi$ is then a homeomorphism, natural in $Y$,
\[
 \left(\map^\Lc(Q,Y)(V)\right)^G \ = \ 
(\map^\Lc(\bL(V\tensor\mR^\infty,\mR^\infty),Y))^G \ \to \ Y^{\bar G}   \ ,\quad
h \ \longmapsto \ h(\psi) \ .
\]
Now we let $f:X\to Y$ be a morphism of~$\Lc$-spaces that is a fibration 
(respectively acyclic fibration) in the $\Cc^L$-projective model structure
of Proposition \ref{prop:projective model structure}.
Then by the above, the morphism of orthogonal spaces~$(\map^\Lc(Q,f)(V))^G$
is isomorphic to the map
\[ f^{\bar G} \ : \ X^{\bar G} \ \to \ Y^{\bar G}\ ,\]
which is a a Serre fibration (respectively a Serre fibration and weak equivalence).
So $\map^\Lc(Q,f)$ is a fibration (respectively acyclic fibration)
in the positive strong level model structure of orthogonal spaces,
compare the proof of \cite[Prop.\,1.2.23]{schwede-global}.
Since the acyclic fibrations agree in the global and 
$\Cc^L$-projective model structure of $\Lc$-spaces, 
and they agree in the positive global and strong level model structures
of orthogonal spaces, this shows in particular that the right adjoint
preserves acyclic fibrations.

Now we suppose that $f:X\to Y$ is a global fibration between injective $\Lc$-spaces.
Then $f$ is in particular a $\Cc^L$-fibration, and hence $\map^\Lc(Q,f)$
is a positive strong level fibration by the previous paragraph.
We let $G$ be a universal subgroup of $\Lc$ and $\varphi:V\to W$ 
a $G$-equivariant linear isometric embedding of $G$-representations,
where $V$ is non-zero and the action on $V$ (and hence also on $W$) is faithful.
We choose a $G$-equivariant
linear isometric embedding $\psi:W\tensor\mR^\infty\to \mR_G^\infty$;
the composite $\kappa=\psi\circ(\varphi\tensor\mR^\infty):V\tensor\mR^\infty\to \mR_G^\infty$
is then also a $G$-equivariant linear isometric embedding.
In the commutative diagram
\[ \xymatrix@C=25mm{ 
(\map^\Lc(Q,X)(V))^G \ar@/^1pc/[rr]^(.7){h\mapsto h(\kappa)}
\ar[r]_-{(\map^\Lc(Q,X)(\varphi))^G} \ar[d]_{(\map^\Lc(Q,f)(V))^G} &
(\map^\Lc(Q,X)(W))^G \ar[d]^{(\map^\Lc(Q,f)(W))^G} \ar[r]_-{h\mapsto h(\psi)}&  
X^G \ar[d]^{f^G}\\
(\map^\Lc(Q,Y)(V))^G\ar@/_1pc/[rr]_(.7){h\mapsto h(\kappa)}
\ar[r]^-{(\map^\Lc(Q,Y)(\varphi))^G} &
(\map^\Lc(Q,Y)(W))^G \ar[r]^-{h\mapsto h(\psi)}& Y^G } \]
the right and composite horizontal maps are weak equivalences
by Proposition \ref{prop:map(Q,Y) and fix}.
So the left horizontal maps 
$(\map^\Lc(Q,X)(\varphi))^G$ and  $(\map^\Lc(Q,Y)(\varphi))^G$
are weak equivalences, and we conclude that the left square is homotopy cartesian.
Since every compact Lie group is isomorphic to a universal subgroup of $\Lc$,
this proves that the morphism $\map^\Lc(Q,f)$ is a fibration in
the positive global model structure of orthogonal spaces.
Now we know that the right adjoint $\map^\Lc(Q,-)$ 
preserves acyclic fibrations, and it takes fibrations between fibrant objects
in the global model structure of $\Lc$-spaces to
fibrations in the positive global model structure of orthogonal spaces.
By a criterion of Dugger \cite[Cor.\,A.2]{dugger-replacing},
this proves that the adjoint functor pair $(Q\tensor_\bL -,\map^\Lc(Q,-))$ 
is a Quillen pair with respect to the two global model structures.

It remains to show that the Quillen pair is a Quillen equivalence.
This is a consequence of the following two facts:
\begin{enumerate}[(a)]
\item The right adjoint $\map^\Lc(Q,-)$
reflects global equivalences between injective $\Lc$-spaces, and
\item for every positive flat orthogonal space $A$, there is a
global equivalence of $\Lc$-spaces  $q:Q\tensor_\bL A\to X$ to an
injective $\Lc$-space such that the composite
\begin{equation} \label{eq:right adjoint criterion}
  A \ \xra{\ \eta_A\ }\ \map^\Lc(Q,Q\tensor_\bL A)\ \xra{\map^\Lc(Q, q)}\ \map^\Lc(Q,X)   
\end{equation}
is a global equivalence of orthogonal spaces.
\end{enumerate}
Indeed, the first property guarantees that the right derived functor
of the right adjoint reflects isomorphisms in the homotopy category,
and the second condition ensures that the unit of the derived
adjunction is a natural isomorphism. So the derived functors
are equivalences of categories, and the Quillen pair is a Quillen equivalence.

(a) We let $f:X\to Y$ be a morphism between injective $\Lc$-spaces
such that $\map^\Lc(Q,f)$ is a global equivalence of orthogonal spaces.
Since $\map^\Lc(Q,-)$ is a right Quillen functor, the orthogonal spaces
$\map^\Lc(Q,X)$ and $\map^\Lc(Q,Y)$ are fibrant in the positive global model structure,
i.e., they are positively static.
So the global equivalence  $\map^\Lc(Q,f)$ is in fact a positive strong level
equivalence by the positive analog of \cite[Prop.\,1.2.14 (i)]{schwede-global}, 
i.e., for every compact Lie group $G$ and every non-zero faithful
$G$-representation $V$ the map
\[ (\map^\Lc(Q,f)(V))^G \ : \ 
(\map^\Lc(Q,X)(V))^G \  \to \ (\map^\Lc(Q,Y)(V))^G \]
is a weak equivalence.
We specialize to the case where $G$ is a universal subgroup of $\Lc$ 
and we choose a faithful, non-zero, finite-dimensional $G$-representation
$V$ and a $G$-equivariant linear isometric embedding
$\kappa:V\tensor\mR^\infty\to \mR^\infty_G$. 
In the commutative square
\[ \xymatrix@C=25mm{ 
(\map^\Lc(Q,X)(V))^G
\ar[r]^-{(\map^\Lc(Q,f)(V))^G} \ar[d]_{h\mapsto h(\kappa)} &
(\map^\Lc(Q,Y)(V))^G \ar[d]^{h\mapsto h(\kappa)} \\
X^G\ar[r]_-{ f^G} & Y^G } \]
the two vertical maps are weak equivalences by Proposition \ref{prop:map(Q,Y) and fix}
and the upper horizontal map is a weak equivalence by the above.
So $f^G$ is a weak equivalence, i.e., $f$ is a global equivalence of $\Lc$-spaces.

(b) Given a positive flat orthogonal space $A$, 
we choose a fibrant replacement $i:A\to B$ in the positive global 
model structure of orthogonal spaces; in other words, $i$ is a positive flat
cofibration and global equivalence, and $B$ is a positively static orthogonal space.
Then we choose a global equivalence of $\Lc$-spaces $p:B(\mR^\infty)\to X$
whose target is injective. 
The composite 
\[ Q\tensor_\bL A \ \xra{Q\tensor_\bL i}\ 
 Q\tensor_\bL B \ \xra{\ \xi_B\ }\ B(\mR^\infty)\ \xra{\ p \ } \ X
 \]
is then the desired global equivalence $q:Q\tensor_\bL A\to X$ with injective target.
This exploits that as a left Quillen functor, $Q\tensor_\bL-$
takes acyclic cofibrations to global equivalences,
and that $\xi_B$ is a global equivalence by Proposition \ref{prop:Q vs O}.
It remains to show that the composite \eqref{eq:right adjoint criterion}
is a global equivalence.

We let $G$ be a universal subgroup of $\Lc$,
and we let $V$ be a finite-dimensional faithful $G$-subrepresentation of
$\mR^\infty_G$. We choose a $G$-equivariant linear
isometric embedding $\kappa:V\tensor \mR^\infty\to\mR^\infty_G$
such that the composite
\[ V\ \xra{-\tensor u}\ V\tensor \mR^\infty\ \xra{\ \kappa \ }\ \mR^\infty_G \]
is the inclusion.
We let $q^\sharp:B\to\map^\Lc(Q,X)$ denote the adjoint of the composite
$p\circ\xi_B:Q\tensor_\bL B\to X$.
Then the following square of $G$-equivariant maps commutes:
\[ \xymatrix@C=7mm{ 
B(V)\ar[rr]^-{q^\sharp(V)}\ar[d]_{B(\text{incl})} && (\map^\Lc(Q,X)(V)) \ar@{=}[r] &
\map^\Lc(\bL(V\tensor\mR^\infty,\mR^\infty),X) \ar[d]^{h\mapsto h(\kappa)}\\
B(\mR^\infty_G)\ar[rrr]_-p &&& X
} \]
The left vertical map induces a weak equivalence of $G$-fixed points
because $B$ is flat, hence closed, and positively static.
The lower horizontal map induces a weak equivalence of $G$-fixed points
because $p$ is a global equivalence and $G$ is a universal subgroup.
The right vertical map induces a weak equivalence on $G$-fixed points
by Proposition \ref{prop:map(Q,Y) and fix}.
So we conclude that the map 
\[ (q^\sharp(V))^G \ : \ B(V)^G \ \to \ \big( \map^\Lc(Q,X)(V) \big)^G \]
is a weak equivalence.
Every pair consisting of a compact Lie group and a finite-dimensional faithful
representation is isomorphic to a pair $(G,V)$ as above.
So the morphism $q^\sharp$ is a positive strong level equivalence of orthogonal spaces,
and hence a global equivalence.
The composite \eqref{eq:right adjoint criterion}
agrees with $q^\sharp\circ i:A\to \map^\Lc(Q,X)$, so it is 
a global equivalence because $i$ and $q^\sharp$ are.

Now we have verified conditions (a) and (b), and this completes the
proof that the adjoint functor pair $(Q\tensor_{\bL}-,\map^\Lc(Q,-))$ 
is a Quillen equivalence.
\end{proof}

\begin{rk}[$\Fc$-global model structure of $\Lc$-spaces] 
Our results all have versions with respect to a 
{\em global family} $\Fc$, i.e., is a class of compact Lie groups that is closed
under isomorphisms, subgroups and quotients.
We only indicate what goes into this, and leave the details to interested readers.
The global model structure of $\Lc$-spaces 
(see Theorem~\ref{thm:global L-spaces}) has a straightforward
version relative to the family~$\Fc$.
We denote by $\Lc\cap \Fc$ the collection of compact Lie subgroups of $\Lc$
that also belong to the global family $\Fc$.
Proposition~\ref{prop:projective model structure}
then provides the  $(\Lc\cap \Fc)$-projective model structure
on the category of $\Lc$-spaces.

We call a morphism $f:X\to Y$ of $\Lc$-spaces an {\em $\Fc$-equivalence}
if for every universal subgroup~$G$ of~$\Lc$ that belongs to~$\Fc$
the induced map
\[ f^G \ : \ X^G \ \to \ Y^G \]
is a weak homotopy equivalence.
We call $f$ an {\em $\Fc$-global fibration} if it is an $(\Lc\cap\Fc)$-fibration
and for every correspondence $\varphi:G\leadsto\bar G$ 
between groups in $\Lc\cap\Fc$ the map
\[ (f^{\bar G}, \varphi\cdot-)\ : \ X^{\bar G} \ \to \ Y^{\bar G} \times_{Y^G }  X^G \]
is a weak equivalence.
Essentially the same proof as for Theorem~\ref{thm:global L-spaces},
but with all relevant groups restricted to the global family $\Fc$,
then shows that the $\Fc$-equivalences, $\Fc$-global fibrations
and $(\Lc\cap\Fc)$-cofibrations form a cofibrantly generated proper 
topological model structure
on the category of $\Lc$-spaces, the {\em $\Fc$-global model structure}.

The Quillen equivalence of Theorem~\ref{thm:L and orbispace}
has a direct analog for a global family $\Fc$, with virtually the same proof.
We let $\bO^\Fc_{\gl}$ be the full topological subcategory
of the global orbit category with objects the universal subgroups of $\Lc$
that belong to the family $\Fc$. Then taking fixed points is a right Quillen equivalence
\[ \Phi^\Fc \ : \  \Lc\bT\ \to \ \Fc\text{-}\orbispc  \]
from the category of $\Lc$-spaces with the $\Fc$-global model structure
to the category of $\Fc$-orbispaces, i.e., contravariant continuous functors
from $\bO_{\gl}^\Fc$ to spaces.
Theorem~\ref{thm:orthogonal vs L} also has a relative version,
with the same proof: for every global family~$\Fc$,
the adjoint functor pair  $(Q\tensor_{\bL}-,\map^\Lc(Q,-))$
is a Quillen equivalence with respect to the positive $\Fc$-global model
structure on orthogonal spaces (the~$\Fc$-based analog 
of~\cite[Prop.\,1.2.23]{schwede-global}) and 
the~$\Fc$-global model structure on $\Lc$-spaces.

When $\Fc=\td{e}$ is the global family of trivial groups, we recover 
the Quillen equivalence established by Lind in~\cite[Thm.\,9.9]{lind-diagram}.
To make the connection, we recall that orthogonal spaces are called
{\em $\mathcal I$-spaces} in~\cite{lind-diagram}, and the category of
$\mathcal I$-spaces is denoted $\mathcal I\mathscr U$. 
Moreover, our~$\Lc$-spaces are called
{\em $\mathbb L$-spaces} in~\cite{lind-diagram}, where ~$\mathbb L$
stands for the monad whose underlying functor sends~$A$ to~$\Lc\times A$;
the category of $\mathbb L$-spaces is denoted $\mathscr U[\mathbb L]$. 
For the trivial global family there is no difference between 
$(\Lc\cap\td{e})$-equivalences and $\td{e}$-equivalences, 
and both specialize to the morphisms of $\Lc$-spaces
that are weak equivalences on underlying non-equivariant spaces.
\end{rk}

\begin{appendix}
\section{Topology of linear isometries}
\label{app A}

In this appendix we collect some facts about the topology on spaces 
of linear isometries; while these facts may be well-known to experts,
not all are particularly well documented in the literature.
I would like to thank Andrew Blumberg and Mike Mandell for key hints
on the proof of Proposition~\ref{prop:properties of L}.

By a `space' we mean a {\em compactly generated space} in the sense of~\cite{mccord},
i.e., a $k$-space (also called {\em Kelley space})
that satisfies the weak Hausdorff condition.
We write $\bT$ for the category  of compactly generated spaces and continuous maps.
We emphasize that in contrast to some other papers on the subject,
the weak Hausdorff condition is subsumed in `compactly generated'.
Besides~\cite{mccord}, some other general references on compactly generated
spaces are the Appendix~A of Lewis' thesis \cite{lewis-thesis},
or Appendix A of the author's book \cite{schwede-global}. 

The category $\bT$ is complete and cocomplete.
Given two compactly generated spaces $X$ and~$Y$, 
we write $X \times Y$ for the product in $\bT$, i.e., the Kelleyfication
of the usual product topology.
If $X$ or $Y$ is locally compact, then the usual product topology
is already a $k$-space, and hence compactly generated,
so in this case no Kelleyfication is needed.
We denote by $\map(X,Y)$ the set of continuous maps from $X$ to $Y$, 
endowed with the Kelleyfication of the compact-open topology.
Then $\map(-,-)$ is the internal function object in $\bT$, 
i.e., for all compactly generated spaces
$X, Y$ and $Z$, the map
\[ \map(X\times Y,Z)\ \to \ \map(X,\map(Y,Z)) \ , \quad
F\ \longmapsto \ \{x\longmapsto F(x,-)\}\]
is a homeomorphism, 
compare \cite[App.\,A, Thm.\,5.5]{lewis-thesis} or \cite[Thm.\,A.23]{schwede-global}.

An {\em inner product space} is an $\mR$-vector space equipped with
a scalar product, i.e., a positive definite symmetric bilinear form.
We will only be concerned with inner product spaces of finite or
countably infinite dimension. A finite-dimensional inner product space
is given the metric topology induced from the scalar product.
An infinite dimensional inner product space is given
the weak topology from the system of its finite-dimensional subspaces.
In infinite dimensions, the weak topology is strictly finer than the metric topology.

If $V$ and $W$ are inner product spaces, we denote by $\bL(V,W)$
the set of linear isometric embeddings, i.e., $\mR$-linear maps
that preserve the inner product. 
A key property of the above topologies 
is that all $\mR$-linear maps are automatically continuous.
So $\bL(V,W)$ is a subset of the space $\map(V,W)$ of continuous maps, and we endow
$\bL(V,W)$ with the subspace topology.
The following proposition shows in particular that $\bL(V,W)$
becomes a compactly generated space in this topology.

\begin{prop}
For all inner product spaces $V$ and $W$, the set $\bL(V,W)$
of linear isometric embeddings is closed in the space $\map(V,W)$
of continuous maps. So $\bL(V,W)$ is a Hausdorff $k$-space in 
the subspace topology from $\map(V,W)$.
\end{prop}
\begin{proof}
Since the evaluation maps, the vector space structure maps
and the norm are continuous, so are the two maps
\[  
\alpha\ : \  \map(V,W) \times\mR\times V\times V \ \to \ W\ ,\quad
 (f,\lambda,v,v') \ \longmapsto \  f(\lambda v + v')-\lambda f(v)-f(v')
\]
and
\[
\beta \ : \  \map(V,W)\times V \ \to \ \mR\ , \quad
(f,v)\ \longmapsto\ |f(v)| - |v| \ . 
\]
Their adjoints
\[ 
\tilde\alpha\ : \  \map(V,W) \ \to \ \map(\mR\times V\times V,W)
 \]
respectively
\[
\tilde \beta \ : \  \map(V,W) \ \to \  \map(V,\mR)
\]
are thus continuous as well.
A map $f:V\to W$ is a linear isometric embedding
if and only if $\tilde\alpha(f)$ and $\tilde\beta(f)$ are the zero maps.
So
\[ \bL(V,W) \ = \ \tilde\alpha^{-1}(0)\cap \tilde\beta^{-1}(0)\]
is a closed subset of $\map(V,W)$.
\end{proof}

Since $\bL(V,W)$ is defined as a subspace of the internal function space,
a map $f:A\to \bL(V,W)$ from a compactly generated space is continuous 
if and only if its adjoint
\[ f^\sharp \ : \ A\times V \ \to \ W \ , \quad  f^\sharp(a,v)\ = \ f(a)(v) \]
is continuous.

\begin{prop}\label{prop:L acts continuously}
For all inner product spaces $U, V$ and $W$, the composition map
\[ \circ \ : \ \bL(V,W) \times \bL(U,V) \ \to \ \bL(U,W) \]
is continuous.
\end{prop}
\begin{proof}
For all compactly generated spaces $X$ and $Y$,
the evaluation map $\ev_{X,Y}:\map(X,Y)\times X\to Y$
is continuous. So  the composite
\begin{equation}\label{eq:adjoint composition}
   \map(V,W) \times \map(U,V) \times U\ \xra{\map(V,W)\times \ev_{U,V}} \
   \map(V,W) \times V\ \xra{\ \ev_{V,W}\ } \ W 
\end{equation}
is continuous. The composition map
\[ \circ \ : \ \map(V,W) \times \map(U,V) \ \to \ \map(U,W) \]
is adjoint to \eqref{eq:adjoint composition}, so it is continuous as well. 
The claim now follows by restricting 
to subspaces of linear isometric embeddings.
\end{proof}

We write~$\mR^\infty$ for the $\mR$-vector space of those sequences
$(x_1,x_2,x_3,\dots)$ of real numbers such that almost all coordinates are zero.
A scalar product on~$\mR^\infty$ is given by the familiar formula
\[ \td{x,y}\  = \ {\sum}_{n\geq 1} \ x_n\cdot y_n \ .  \]
So~$\mR^\infty$ is of countably infinite dimension, and the standard basis
(i.e., the sequences with one coordinate~1 and all other coordinates~0) is an orthonormal basis. 
We identify $\mR^n$ with the subspace of $\mR^\infty$ consisting 
of those tuples $(x_1,x_2,x_3,\dots)$ such that $x_i=0$ for all $i> n$;
then $\mR^\infty$ carries the weak topology with respect to
the nested subspaces $\mR^n$.

We will also want to know that whenever $V$ is finite-dimensional, 
then the function space topology on $\bL(V,W)$
is in fact the familiar `Stiefel manifold topology'. 
The following proposition is needed for this identification;
it is a special case of Proposition 9.5 in \cite[App.\,A]{lewis-thesis}.

\begin{prop}\label{prop:mapping commute}
For every compact topological space $K$ the canonical map  
\[ \kappa \ : \ \colim_{n\geq 0}  \map(K,\mR^n)\ \to \ \map(K,\mR^\infty)\]
is a homeomorphism.
\end{prop}
\begin{proof}
Since $\mR^\infty$ is a colimit of the sequence of closed embeddings
$\mR^n\to\mR^{n+1}$ between compactly generated spaces, 
every continuous map $f:K\to \mR^\infty$ factors  through $\mR^n$ for some $n\geq 0$,
compare \cite[App.\,A, Lemma 9.4]{lewis-thesis} or \cite[Prop.\,A.15 (i)]{schwede-global}.
So the canonical map $\kappa$ is a continuous bijection.

It remains to show that the inverse of the canonical map is continuous.
Since source and target of $\kappa$ are $k$-spaces, continuity of $\kappa^{-1}$
can be probed by continuous maps from compact spaces,
i.e., it suffices to show that for every continuous map $f:L\to \map(K,\mR^\infty)$
from a compact space the composite $\kappa^{-1}\circ f$ is continuous.
For this it is enough to show that $f$ factors through $\map(K,\mR^n)$ for some $n\geq 0$.

This, however, is now easy: we let $f^\sharp:L\times K\to \mR^\infty$ 
be the continuous adjoint of $f$, defined by $f^\sharp(l,k)=f(l)(k)$. 
Since $L$ and $K$ are compact, so is $L\times K$, and so $f^\sharp$ 
factors through $\mR^n$ for  some $n\geq 0$.
Hence $f$ factors through $\map(K,\mR^n)$, and this completes the proof.
\end{proof}

\begin{prop}\label{prop:properties of L}
Let $V$ and $W$ be inner product spaces.
\begin{enumerate}[\em (i)]
  \item If $V$ and $W$ are finite-dimensional, then
    $\bL(V,W)$ is homeomorphic to the Stiefel manifold 
    of $\dim(V)$-frames in $W$, and hence compact.
  \item If $V$ is finite-dimensional, then
    $\bL(V,\mR^\infty)$ carries the weak topology with respect to
    the filtration by the compact closed subspaces $\bL(V,\mR^n)$.
  \item
    Let 
    \[ V_1\ \subset\ V_2\ \subset\ \dots\ \subset\ V_n\ \subset\ \dots  \]
    be a nested exhausting sequence of finite-dimensional subspaces of $\mR^\infty$.
    Then $\bL(\mR^\infty,\mR^\infty)$ is an inverse limit, in the category of compactly
    generated spaces, of the tower of restriction maps
    \[
    \dots\ \to\ \bL(V_n,\mR^\infty)\ \to\ \dots \ \to\  \bL(V_2,\mR^\infty)\ \to \ 
    \bL(V_1,\mR^\infty)\ . \]
  \end{enumerate}
\end{prop}
\begin{proof}
(i) We denote by $\bL^{St}(V,W)$ the set of linear isometric embeddings
endowed with the Stiefel manifold topology, making it a closed manifold.
The evaluation map
  \[ \bL^{St}(V,W)\times V \ \to \ W \]
is continuous, hence so is its adjoint, the inclusion
\[ \bL^{St}(V,W) \ \to \ \map(V,W)\ . \]
Since $\bL^{St}(V,W)$ is compact and $\map(V,W)$ is a Hausdorff space,
the inclusion is a closed embedding. So the Stiefel manifold
topology coincides with the subspace topology of the function topology.

(ii)
We let $A$ be a subset of $\bL(V,\mR^\infty)$ such that
$A\cap \bL(V,\mR^n)$ is closed in $\bL(V,\mR^n)$ for all $n\geq 0$;
we wish to show that then $A$ is itself closed.
For this purpose we let $\rho:\bL(V,\mR^\infty)\to\map(S(V),\mR^\infty)$
denote the composite
\[ 
 \bL(V,\mR^\infty)\ \xra{\text{incl}}\ \map(V,\mR^\infty) \ 
\ \xra{\text{restr}}\ \map(S(V),\mR^\infty) \ ,
 \]
where the last map is restriction to the unit sphere $S(V)$.
We observe that
\begin{equation} \label{eq:intersection}
 \rho(A)\cap \map(S(V),\mR^n)\ = \ \rho_n(A\cap \bL(V,\mR^n)) \ ,  
\end{equation}
where
\[ \rho_n\ : \  \bL(V,\mR^n)\ \to \ \map(S(V),\mR^n)\]
is restriction to the unit sphere.
The space $\bL(V,\mR^n)$ is compact by part (i); so the closed subset
$A\cap \bL(V,\mR^n)$ is itself compact.
Then $\rho_n(A\cap\bL(V,\mR^n))$ is a quasi-compact subset of the Hausdorff space
$\map(S(V),\mR^n)$, and hence closed.
So $\rho(A)\cap \map(S(V),\mR^n)$ is closed for all $n\geq 0$, by \eqref{eq:intersection}.
Since the unit sphere $S(V)$ is compact, 
Proposition~\ref{prop:mapping commute} shows that the space
$ \map(S(V),\mR^\infty)$ has the weak topology
with respect to the filtration by the closed subspaces  $\map(S(V),\mR^n)$.
So the set $\rho(A)$ is closed in $\map(S(V),\mR^\infty)$.
Since $\rho$ is continuous and injective, the set
\[ A \ = \ \rho^{-1}(\rho(A)) \]
is closed in $\bL(V,\mR^\infty)$.

(iii)
The category of compactly generated spaces is cartesian closed, with $\map(-,-)$ 
as the internal function object. 
So the natural bijections
\begin{align*}
\bT^{\op}(\map(Y,Z),X)\ = \ \bT(X,\map(Y,Z))\ \iso \ 
\bT(X\times Y,Z)\ \iso \ \bT(Y,\map(X,Z))  
\end{align*}
show that for every compactly generated space $Z$, the functor
\[ \map(-,Z)\ : \ \bT\ \to \ \bT^{\op} \]
is a left adjoint. So $\map(-,Z)$ takes colimits in $\bT$ to 
colimits in $\bT^{\op}$, which are limits in $\bT$.
Since $\mR^\infty$ is a colimit of the sequence $\{V_n\}_{n\geq 0}$,
the space $\map(\mR^\infty,\mR^\infty)$ is an inverse limit of the tower
of spaces $\{\map(V_n,\mR^\infty)\}_{n\geq 0}$.

Now we can prove the claim. We let $f_n:T\to \bL(V_n,\mR^\infty)$
be a compatible family of continuous maps from a compactly generated space~$T$.
Then the composite maps
\[ T \ \xra{\ f_n \ }\ \bL(V_n,\mR^\infty) \ \xra{\text{incl}}\ \map(V_n,\mR^\infty)\]
are compatible. Since $\map(\mR^\infty,\mR^\infty)$
is an inverse limit of the tower, there is a unique continuous map
$f:T\to \map(\mR^\infty,\mR^\infty)$ such that the square
\[ \xymatrix{ 
T\ar[r]^-f \ar[d]_{f_n} & \map(\mR^\infty,\mR^\infty)\ar[d]^{\text{restr}} \\
\bL(V_n,\mR^\infty) \ar[r]_-{\text{incl}} & \map(V_n,\mR^\infty)}\]
commutes for all $n\geq 0$.
For every $t\in T$ the map $f(t):\mR^\infty\to\mR^\infty$
restricts to a linear isometric embedding on every finite-dimensional
subspace of $\mR^\infty$; so $f(t)$ is itself a linear isometric embedding.
Hence the map $f$ lands in the subspace $\bL(\mR^\infty,\mR^\infty)$
of $\map(\mR^\infty,\mR^\infty)$, and can be viewed as a continuous map
$f:T\to \bL(\mR^\infty,\mR^\infty)$ that restricts to the original map $f_n$
for all $n\geq 0$. This proves that  $\bL(\mR^\infty,\mR^\infty)$
has the universal property of an inverse limit 
of the tower of restriction maps
$\bL(V_n,\mR^\infty) \to\bL(V_{n-1},\mR^\infty)$.
\end{proof}

In this paper we are very much interested in subgroups of the linear
isometries monoid that happen to be compact Lie groups. 
The following proposition collects some useful facts about
continuous actions of compact groups by linear isometries on $\mR^\infty$.

\begin{prop}\label{prop:compact subgroups}
Let $G$ be a compact topological group acting continuously on $\mR^\infty$
by linear isometries.
\begin{enumerate}[\em (i)]
\item Every finite-dimensional subspace of $\mR^\infty$
is contained in a finite-dimensional $G$-invariant subspace.
\item The space $\mR^\infty$ is the orthogonal direct sum of finite-dimensional
$G$-invariant subspaces.
\end{enumerate}
\end{prop}
\begin{proof}
(i) We let $V$ be a finite-dimensional subspace of $\mR^\infty$
and denote by
\[ \rho \ : \ G \ \to \ \bL(V,\mR^\infty) \]
the adjoint to the action map 
\[ G\times V\ \to \ \mR^\infty \ , \quad (g,v)\ \longmapsto \ g v\ .\]
Since $G$ acts continuously on $\mR^\infty$, the adjoint $\rho$
is continuous with respect to the function space topology.
By Proposition \ref{prop:properties of L}~(ii), the space
$\bL(V,\mR^\infty)$ carries the weak topology with respect to
the nested sequence of closed subspaces $\bL(V,\mR^n)$. 
So the compact subset $\rho(G)$ is contained in $\bL(V,\mR^n)$
for some $n\geq 0$. In more concrete terms, this means that
 $G\cdot V\subseteq \mR^n$.
The $\mR$-linear span of the set $G\cdot V$ is then
the desired finite-dimensional $G$-invariant subspace.

(ii) 
We construct pairwise orthogonal finite-dimensional $G$-invariant subspaces $W_n$
such that $\mR^n$ is contained in the direct sum of $W_1,\dots,W_n$
for all $n\geq 1$.
The construction is inductive, starting with $W_0=0$. 
The inductive step uses part (i) to obtain a finite-dimensional 
$G$-invariant subspace $V$ 
of $\mR^\infty$ that contains $W_1,\dots, W_{n-1}$ and $\mR^n$.
Then we let $W_n$ be the orthogonal complement of $W_1\oplus\dots\oplus W_{n-1}$ in $V$.
Since $W_1,\dots,W_{n-1}$ are $G$-invariant and $G$ acts by linear isometries,
$W_n$ is again $G$-invariant.
Moreover, since $V$ is finite-dimensional, it is 
the orthogonal direct sum of $W_1\oplus\dots\oplus W_{n-1}$ and $W_n$.
Since the finite sums $W_1\oplus\dots\oplus W_n$  exhaust $\mR^\infty$, the latter is the
orthogonal direct sum of all the subspaces $W_n$ for $n\geq 1$.
\end{proof}

\begin{prop}\label{prop:characterize Lie}
Let $G$ be a compact subgroup of the linear isometries monoid $\Lc$.
Then $G$ admits the structure of a Lie group (necessarily unique)
if and only if there is a finite-dimensional 
$G$-invariant subspace of $\mR^\infty$ on which $G$ acts faithfully.
\end{prop}
\begin{proof}
We start by assuming that there is 
a $G$-invariant finite-dimensional subspace $V$ of $\mR^\infty$
on which $G$ acts faithfully.
Then the continuous composite
\[ G \ \xra{\text{incl}}  \ \Lc \ = \ \bL(\mR^\infty,\mR^\infty)\
\xra{\text{restr}}\ \bL(V,\mR^\infty) \]
factors through an injective continuous group homomorphism
\[ \rho\ : \ G \ \to \ \bL(V,V) \]
that encodes the $G$-action on $V$.
By Proposition \ref{prop:properties of L}~(i), the space
$\bL(V,V)$ carries the Stiefel manifold topology, i.e., 
$\bL(V,V)=O(V)$ is the orthogonal group of $V$ with the Lie group
topology. Since $G$ is compact, the homomorphism $\rho:G\to O(V)$
is a closed map, hence a homeomorphism onto its image, which is a closed
subgroup of $O(V)$. Every closed subgroup of a Lie group carries
the structure of a Lie group, compare \cite[Ch.\,I, Thm.\,3.11]{broecker-tomDieck};
so $G$ admits the structure of a Lie group. 
A topological group admits at most one structure of Lie group, see for example 
\cite[Ch.\,I, Prop.\,3.12]{broecker-tomDieck}; 
so the Lie group structure is necessarily unique.

Now we suppose that conversely $G$ admits the structure of a Lie group.
Proposition \ref{prop:compact subgroups}~(ii) 
provides finite-dimensional pairwise orthogonal $G$-invariant subspaces $W_n$ that
together span $\mR^\infty$.
We let 
\[ H_n \ = \ \{g\in G\ : \ g v = v \text{ for all $v\in W_1\oplus \dots\oplus W_n$}\} \]
denote the kernel of the $G$-action on $W_1\oplus\dots\oplus W_n$;
this is a closed normal subgroup of $G$ with $H_{n+1}\subseteq H_n$. 
In a compact Lie group, every infinite
descending chain of closed subgroups is eventually constant.
So there is an $m\geq 0$ such that $H_n=H_m$ for all $n\geq m$.
Since the subspaces $W_n$ span $\mR^\infty$, the subgroup $H_m$ is the trivial subgroup.
Hence $W_1\oplus\dots\oplus W_m$ 
is the desired finite-dimensional faithful $G$-invariant subspace of $\mR^\infty$.
\end{proof}

\begin{eg}\label{eg:not Lie}
The linear isometries monoid $\Lc$ has compact subgroups that are not Lie groups,
and for which $\mR^\infty$ has no faithful finite-dimensional subrepresentation.
As an example we consider the unitary representation $\mC(n)$ of the 
additive group $\mZ^\wedge_p$ of $p$-adic integers
on the complex numbers through the finite quotient $\mZ/p^n\mZ$,
with a generator of $\mZ/p^n\mZ$ acting by multiplication by $e^{2\pi i/ p^n}$.
Then the direct sum action of $\mZ^\wedge_p$ on 
$\bigoplus_{n\geq 1} \mC(n)$ is faithful and through $\mC$-linear isometries.
An identification of the underlying $\mR$-vector space of 
$\bigoplus_{n\geq 1} \mC(n)$ with $\mR^\infty$ turns this into a faithful,
continuous isometric action of $\mZ^\wedge_p$ on $\mR^\infty$.
The image of this action is then a compact subgroup of $\Lc$ isomorphic to $\mZ^\wedge_p$.
\end{eg}

\begin{prop}\label{prop:EKG infinite} 
Let $G$ and $K$ be compact Lie groups,
$\Vc$ an orthogonal $G$-representation of countably infinite dimension 
and $\Uc_K$ a complete $K$-universe.
For every faithful finite-dimensional $G$-subrepresentation~$V$ of $\Vc$ 
the restriction morphism
\[ \rho^{\Vc}_V  \ : \ \bL(\Vc,\Uc_K) \ \to \ \bL(V,\Uc_K)  \]
is a $(K\times G)$-homotopy equivalence.
\end{prop}
\begin{proof}
We choose an exhausting nested sequence
\[ V\ = \ V_0 \ \subset \ V_1 \ \subset \ V_2 \ \dots \]
of finite-dimensional $G$-subrepresentations of $\Vc$, starting with the
given faithful representation.
We claim that all the restriction maps
\[ p_n \ : \ \bL(V_n,\Uc_K)\ \to \ \bL(V_{n-1},\Uc_K) \]
are $(K\times G)$-acyclic fibrations, i.e., for every closed subgroup
$\Gamma\leq K\times G$ the fixed point map
\[ (p_n)^\Gamma \ : \ \bL(V_n,\Uc_K)^\Gamma\ \to \ \bL(V_{n-1},\Uc_K)^\Gamma \]
is a weak equivalence and Serre fibration.
Since $G$ acts faithfully on $V_n$,
the $\Gamma$-fixed points of source and target are empty
whenever $\Gamma\cap (1\times G)\ne\{(1,1)\}$ .
Otherwise~$\Gamma$ is the graph 
of a continuous homomorphism $\alpha:L\to G$ defined on a closed subgroup $L$ of $K$.
So the fixed point map is the restriction map
\[ (p_n)^{\Gamma} \ : \ \bL^L(\alpha^*(V_n),\Uc_K)\ \to \ 
\bL^L(\alpha^*(V_{n-1}),\Uc_K)  \ .\]
Source and target of this map are contractible 
(for example by~\cite[II Lemma 1.5]{lms}), so the map
$(p_n)^{\Gamma}$ is a weak equivalence. 
But $(p_n)^{\Gamma}$ is also a locally trivial fiber bundle, 
hence a Serre fibration.

Now we know that $p_n$ is a $(K\times G)$-acyclic fibration, 
and moreover $\bL(V_n,\Uc_K)$ is $(K\times G)$-cofibrant for every $n\geq 0$, 
for example by~\cite[Prop.\,1.1.19 (ii)]{schwede-global}.
Moreover, the space $\bL(\Vc,\Uc_K)$ is the inverse limit
of the tower of restriction maps $p_n : \bL(V_n,\Uc_K)\to \bL(V_{n-1},\Uc_K)$,
by Proposition \ref{prop:properties of L}. 
So the following Proposition~\ref{prop:inverse limit homotopy},
applied to the projective model structure on $(K\times G)$-spaces
(compare \cite[Prop.\,B.7]{schwede-global})
shows that the restriction map 
\[ \rho^{\Vc}_V  \ : \ \bL(\Vc,\Uc_K) \ \to \ \bL(V,\Uc_K)  \]
is a $(K\times G)$-homotopy equivalence.
\end{proof}

\begin{prop}\label{prop:inverse limit homotopy}
  Let~$\Cc$ be a topological model category and 
\[  \dots \ \xra{} \ X_n \ \ \xra{\ p_n\ } \ \dots \ \xra{} \ X_2 \ \xra{\ p_2\ } 
\ X_1 \ \xra{\ p_1\ } \ X_0 \]
an inverse system of acyclic fibrations between cofibrant objects.
Then the canonical map 
\[ p_\infty\ : \ X_\infty \ = \ \lim_n \, X_n \ \to \ X_0 \]
is a homotopy equivalence.
\end{prop}
\begin{proof}
Since $p_n$ is an acyclic fibration and $X_{n-1}$ is cofibrant, 
we can find a section $s_n:X_{n-1}\to X_n$ to~$p_n$ by choosing a lift in the
left square below:
\[ \xymatrix@C=15mm{ 
\emptyset \ar[r]\ar[d] & X_n\ar[d]^{p_n} &
X_n\times\{0,1\}\ar[d]\ar[r]^-{\Id + s_n\circ p_n} & X_n\ar[d]^{p_n}\\
X_{n-1}\ar@{=}[r]\ar@{-->}[ur]^(.4){s_n} &X_{n-1} &
X_n\times[0,1]\ar[r]_-{p_n\circ \text{proj}}\ar@{-->}[ur]^(.4){H_n}& X_{n-1} 
} \]
Since $X_n$ is cofibrant, the inclusion $X_n\times\{0,1\}\to X_n\times[0,1]$
is a cofibration, and a choice of lift in the right diagram provides a homotopy
\[ H_n \ : \ X_n\times [0,1]\ \to \ X_n \]
from the identity to $s_n\circ p_n$ such that 
$p_n\circ H_n:X_n\times [0,1]\to X_{n-1}$
is the constant homotopy from~$p_n$ to itself.
The morphisms
\[ s_n\circ s_{n-1}\circ \cdots \circ s_1 \ : \  X_0\ \to \ X_n \]
are then compatible with the inverse system defining $X_\infty$, 
so they assemble into a morphism
\[ s_\infty\ : \ X_0\ \to \ {\lim}_n\, X_n \]
to the inverse limit, and $s_\infty$ is a section to $p_\infty$.

We claim that the composite~$s_\infty\circ p_\infty$ is homotopic to the identity.
To prove the claim we construct compatible homotopies
\[ K_n \ : \ X_\infty \times [0,1] \ \to \ X_n  \]
by induction on~$n$ satisfying
\begin{enumerate}[(i)]
\item $p_n\circ K_n=K_{n-1}$, 
\item $K_n(-,t)=  p^{(n)}_\infty$, the canonical morphism $X_\infty\to X_n$,
for all $t\in [0,\genfrac{}{}{}{}{1}{n+1}]$, and
\item $K_n(-,1)= s_n\circ s_{n-1}\circ \cdots s_1 \circ p_\infty$.
\end{enumerate}
The induction starts by defining $K_0$
as the constant homotopy from $p_\infty:X_\infty\to X_0$ to itself.
Now we assume $n\geq 1$ and suppose that the homotopies $K_0,\dots,K_{n-1}$
have already been constructed. 
We exploit that the functor $X_\infty\times-$ is a left adjoint,
so $X_\infty\times[0,1]$ is the pushout of the objects
$X_\infty\times[0,\genfrac{}{}{}{}{1}{n+1}]$,
$X_\infty\times[\genfrac{}{}{}{}{1}{n+1},\genfrac{}{}{}{}{1}{n}]$ and
$X_\infty\times [\genfrac{}{}{}{}{1}{n},1]$ along two copies of $X_\infty$,
embedded via 
the points $\genfrac{}{}{}{}{1}{n+1}$ and $\genfrac{}{}{}{}{1}{n}$ of $[0,1]$.
So we can define~$K_n$ by
\[ K_n(-,t) \ = \
\begin{cases}
\qquad p^{(n)}_\infty  & \text{\ for $t\in [0,\genfrac{}{}{}{}{1}{n+1}]$,} \\
H_n(-,\, n(n+1)t- n)\circ p^{(n)}_\infty  & 
\text{\ for $t\in [\genfrac{}{}{}{}{1}{n+1},\genfrac{}{}{}{}{1}{n}]$, and} \\  
\quad s_n\circ K_{n-1}(-,t) & \text{\ for $t\in [\genfrac{}{}{}{}{1}{n},1]$.} 
\end{cases}
\]
This is well-defined at the intersections of the intervals because
\[ H_n\left(-,\, n(n+1)\genfrac{}{}{}{}{1}{n+1}- n\right)\circ p^{(n)}_\infty \ = \ 
 H_n(-,0)\circ p^{(n)}_\infty  \ = \ p^{(n)}_\infty \]
and
\begin{align*}
   H_n\left(-,\, n(n+1)\genfrac{}{}{}{}{1}{n}- n\right)\circ p^{(n)}_\infty \ &= \ 
 H_n(-,1)\circ p^{(n)}_\infty  \ = \ s_n\circ p_n\circ p^{(n)}_\infty\\
&= \ s_n\circ  p^{(n-1)}_\infty\ = \ s_n\circ  K_{n-1}(-, 1/n) 
\end{align*}
Then condition~(i) holds because
\begin{align*}
p_n\circ K_n(-,t) \ &= \
\begin{cases}
\ p_n\circ p^{(n)}_\infty  & \text{\ for $t\in [0,\genfrac{}{}{}{}{1}{n+1}]$,} \\
\ p_n\circ  H_n(-,\, n(n+1)t- n)\circ p^{(n)}_\infty   & \text{\ for $t\in [\genfrac{}{}{}{}{1}{n+1},\genfrac{}{}{}{}{1}{n}]$, and} \\  
\ p_n\circ s_n \circ K_{n-1}(-,t) & \text{\ for $t\in [\genfrac{}{}{}{}{1}{n},1]$,} 
\end{cases}\\
&= \
\begin{cases}
\ p^{(n-1)}_\infty  & \text{\ for $t\in [0,\genfrac{}{}{}{}{1}{n}]$,} \\
\ K_{n-1}(-,t) & \text{\ for $t\in [\genfrac{}{}{}{}{1}{n},1]$,} 
\end{cases}\\
&= \ K_{n-1}(-,t) \ .
\end{align*}
Now we can finish the proof. By condition~(i) the homotopies~$K_n$ 
are compatible, so they assemble 
into a morphism $K_\infty:X_\infty\times [0,1]\to X_\infty$.
Property~(ii) shows that $K_\infty$ starts with the identity of $X_\infty$
and property~(iii) ensures that $K_\infty$ ends 
with the morphism  $s_\infty\circ p_\infty$.
So $s_\infty$ is a homotopy inverse to~$p_\infty$.
\end{proof}

\begin{rk}\label{rk:L contractible}
If we specialize Proposition~\ref{prop:EKG infinite} 
to $\Vc=\Uc_K=\mR^\infty$ and ignore all group actions, it shows in particular
that the restriction map $\Lc=\bL(\mR^\infty,\mR^\infty)\to \bL(0,\mR^\infty)$
is a homotopy equivalence to a one-point space. In other words,
the underlying space of $\Lc$ is contractible.
\end{rk}

\end{appendix}

\end{document}